\documentclass[12pt]{amsart}
\usepackage{amscd,verbatim}
\usepackage{amssymb}
\usepackage[T1]{fontenc}
\usepackage[all]{xy}
\usepackage[%dvipdfm,
colorlinks,linkcolor=blue,citecolor=blue,urlcolor=red]{hyperref}

\newcommand{\A}{\mathbf{A}}

\renewcommand{\P}{\mathbf{P}}

\newcommand{\Z}{\mathbb{Z}}
\newcommand{\sA}{\mathcal{A}}

\newcommand{\sC}{\mathcal{C}}
\newcommand{\sD}{\mathcal{D}}

\newcommand{\sU}{\mathcal{U}}
\newcommand{\sV}{\mathcal{V}}
\newcommand{\sW}{\mathcal{W}}
\newcommand{\sX}{\mathcal{X}}

\newcommand{\Cor}{\operatorname{\mathbf{Cor}}}

\newcommand{\ul}[1]{{\underline{#1}}}

\newcommand{\DM}{\operatorname{\mathbf{DM}}}

\newcommand{\ulMDM}{\operatorname{\mathbf{\underline{M}DM}}}
\newcommand{\MDM}{\operatorname{\mathbf{MDM}}}
\newcommand{\Wei}{{\operatorname{Wei}}}
\newcommand{\Car}{{\operatorname{Car}}}

\newcommand{\Ker}{\operatorname{Ker}}

\newcommand{\Spec}{\operatorname{Spec}}

\newcommand{\Sm}{\operatorname{\mathbf{Sm}}}
\newcommand{\Sch}{\operatorname{\mathbf{Sch}}}
\newcommand{\Ab}{\operatorname{\mathbf{Ab}}}

\newcommand{\by}{\xrightarrow}

\newcommand{\iso}{\by{\sim}}

\newcommand{\pro}[1]{\text{\rm pro}_{#1}\text{\rm--}}
\newcommand{\tr}{{\operatorname{tr}}}

\newcommand{\mor}{{\operatorname{mor}}}

\newcommand{\eff}{{\operatorname{eff}}}
\newcommand{\fin}{{\operatorname{fin}}}
\renewcommand{\o}{{\operatorname{o}}}
\newcommand{\op}{{\operatorname{op}}}

\newcommand{\red}{{\operatorname{red}}}

\newcommand{\Nis}{{\operatorname{Nis}}}
\newcommand{\et}{{\operatorname{\acute{e}t}}}
\newcommand{\tto}{\dashrightarrow}
\newcommand{\inj}{\hookrightarrow}

\newcommand{\Tot}{\operatorname{Tot}}

\renewcommand{\lim}{\operatornamewithlimits{\varprojlim}}
\newcommand{\colim}{\operatornamewithlimits{\varinjlim}}

\newcommand{\ol}{\overline}

\renewcommand{\phi}{\varphi}
\renewcommand{\epsilon}{\varepsilon}

\newcommand{\gm}{{\operatorname{gm}}}

\newcommand{\MNST}{\operatorname{\mathbf{MNST}}}

\newcommand{\MCor}{\operatorname{\mathbf{MCor}}}
\newcommand{\MP}{\operatorname{\mathbf{MSm}}}

\newcommand{\MPST}{\operatorname{\mathbf{MPST}}}

\newcommand{\Bl}{{\mathbf{Bl}}}

\newcommand{\Cat}{\operatorname{\mathbf{Cat}}}

\newcommand{\bcube}{{\ol{\square}}}

\newcommand{\Funct}{\operatorname{\mathbf{Funct}}}

\newcommand{\ulMP}{\operatorname{\mathbf{\underline{M}Sm}}}

\newcommand{\ulMPST}{\operatorname{\mathbf{\underline{M}PST}}}
\newcommand{\ulMNST}{\operatorname{\mathbf{\underline{M}NST}}}
\newcommand{\ulMCor}{\operatorname{\mathbf{\underline{M}Cor}}}
\newcommand{\ulomega}{\underline{\omega}}

\newcommand{\Comp}{\operatorname{\mathbf{Comp}}}

\newcommand{\Sq}{{\operatorname{\mathbf{Sq}}}}

\newcounter{spec}
\newenvironment{thlist}{\begin{list}{\rm{(\roman{spec})}}%
{\usecounter{spec}\labelwidth=20pt\itemindent=0pt\labelsep=10pt}}%
{\end{list}}%

\newtheorem{Th}{Theorem}
\newtheorem{lemma}{Lemma}[section]
\newtheorem{thm}[lemma]{Theorem}

\newtheorem{prop}[lemma]{Proposition}

\newtheorem{cor}[lemma]{Corollary}

\theoremstyle{definition}
\newtheorem{defn}[lemma]{Definition}

\newtheorem{definition}[lemma]{Definition}

\theoremstyle{remark}

\newtheorem{rk}[lemma]{Remark}
\newtheorem{remark}[lemma]{Remark}

\newtheorem{example}[lemma]{Example}
\newtheorem{claim}[lemma]{Claim}

\numberwithin{equation}{section}

\begin{document}

\title{Mayer-Vietoris triangles for motives with modulus}
\author{Bruno Kahn}
\address{IIMJ-PRG\\ Case 247\\4 place
Jussieu\\75252 Paris Cedex 05\\France}
\email{bruno.kahn@imj-prg.fr}
\author{Hiroyasu Miyazaki}
\address{IIMJ-PRG\\ Case 247\\4 place
Jussieu\\75252 Paris Cedex 05\\France}
\email{hiroyasu.miyazaki@imj-prg.fr}
\date{September 15, 2018}
\begin{abstract}
We construct ``MV squares'' in the category $\MCor$ of modulus pairs which was introduced in \cite{motmod}. They allow us to describe the category $\MDM_\gm^\eff$ of loc. cit. in a similar way as Voevodskys category $\DM_\gm^\eff$ in \cite{voetri}, thus sharpening the results of \cite{motmod}.
\end{abstract}
\maketitle

\tableofcontents

\section*{Introduction}

In \cite{voetri}, Voevodsky defines his triangulated category $\DM_\gm^\eff$ of geometric motives over a field $k$ ``by generators and relations'': generators are motives $M(X)$ of smooth $k$-varieties $X$, and relations are of two kinds: 
\begin{description}
\item[$\A^1$-invariance] $M(X\times \A^1)\iso M(X)$ for any $X$;
\item [Mayer-Vietoris  exact triangles] for any Zariski cover $X=U \cup V$, the sequence
\[M(U\cap V)\to M(U)\oplus M(V)\to M(X)\]
yields an exact triangle in $\DM_\gm^\eff$.
\end{description}

When $k$ is perfect, one gets more general exact triangles, associated to elementary  Nisnevich covers. This is a highly non-trivial theorem of Voevodsky, and it is more reasonable to refound his theory by imposing these latter relations even when the field $k$ is not perfect: this renders part of this theory more elementary \cite[\S 4]{2017.1}. 

This is the approach which is adopted in \cite{motmod} to construct a triangulated category $\MDM_\gm^\eff$ of ``motives with modulus''. Unfortunately, the situation is not so simple. Namely, one first constructs in \cite[\S 6.2]{motmod} a larger triangulated category $\ulMDM_\gm^\eff$ ``\`a la Voevodsky'': its generators are motives of modulus pairs whose total space is not necessarily proper, and relations are parallel to those of Voevodsky:
\begin{description}
\item[$\bcube$-invariance] $M(\sX\otimes \bcube)\iso M(\sX)$ for any modulus pair $\sX$;
\item[Mayer-Vietoris  exact triangles] for any elementary Nisnevich cover
\begin{equation}
\begin{CD}
\sW @>>> \sV\\
@VVV @VVV\\
\sU@>>> \sX
\end{CD}
\end{equation}
 the sequence
\[M(\sW)\to M(\sU)\oplus M(\sV)\to M(\sX)\]
yields an exact triangle in $\ulMDM_\gm^\eff$.
\end{description}

Here, $\bcube$ is the modulus pair $(\P^1,\infty)$ (completing $\A^1$), and ``elementary Nisnevich covers'' are defined in a suitable category of modulus pairs, in a way parallel to the classical case. The category $\MDM_\gm^\eff$ is then the full triangulated subcategory of $\ulMDM_\gm^\eff$ generated by motives of proper modulus pairs \cite[\S 6.9]{motmod}.

In $\MDM_\gm^\eff$, the $\bcube$-invariance relation makes sense, because the modulus pair $\bcube$ is proper. This is not true for the Mayer-Vietoris relation, because the use of elementary Nisnevich covers forces us to leave the world of proper varieties. Nevertheless, in \cite[\S 7.5]{motmod} we exhibit exact triangles in $\MDM_\gm^\eff$ ``of Mayer-Vietoris type'' in a certain sense.

Are there enough such triangles to present $\MDM_\gm^\eff$ in these terms? The main result of the present paper is a positive answer to this question. More precisely, write $\MCor$ for the category of (proper) modulus pairs, as in \cite[Def. 1.3.1]{motmod}, and $\MPST$ for the category of $\MCor$-modules as in \cite[Def. 2.1.1]{motmod}. %Let $CI$ mean ``$\bcube$-invariance'' and $MV$ mean the condition of \cite[Def. 4.3.7]{motmod} (the latter will be recalled below). 
Then we have:

\begin{Th}\label{th-presentation}  Let $CI$ be the collection of complexes of the form $\sX\otimes \bcube\to \sX$ for $\sX\in \MCor$, and let  $MV$ denote the collection of complexes of the form
\[N(00)\to N(10)\oplus N(01)\to N(11)\]
where $\ul{N}=\{N(ij)\mid i,j\in \{0,1\}\}$ is an MV square in $\MCor$ as in Definition \ref{d2.3} a) below. Then in the naturally commutative square
\[\begin{CD}
\left(K^b(\MCor)/\langle CI + MV\rangle\right)^\natural@>\alpha>> D(\MPST)/\langle \Z_\tr(CI) + \Z_\tr(MV)\rangle^{loc}\\
@V\gamma VV @V\delta VV\\
\MDM_\gm^\eff@>\beta>> \MDM^\eff
\end{CD}\]
both vertical functors are equivalences of categories. Here, $\langle\ \rangle^{loc}$ means ``localising subcategory generated by''.
\end{Th}

The proof of Theorem \ref{th-presentation} can be sketched as follows. 
We first reduce it to a cofinality statement, Theorem \ref{c1} (see Prop. \ref{p2.2}). In \cite[\S 4.3]{motmod}, a strategy was implicitly suggested to prove Theorem \ref{c1}. We broadly follow this strategy, but the story turns out to be more complicated. Namely, we reduce in Proposition \ref{p2.1} the proof of Theorem \ref{c1} to two statements, Theorems \ref{th-partial} and \ref{existence-MV}. They are respectively Nisnevich generalizations of \cite[Lemma 4.3.5]{motmod} and \cite[Prop. 4.3.10]{motmod} (which only concern elementary Zariski squares), except that each one is much stronger. 

 Let us now give some ideas of the proofs of Theorems \ref{th-partial} and \ref{existence-MV}. 
Since their statements are quite technical, we offer here simplified (but weaker) statements.
This helps explain our strategy.

Theorem \ref{th-partial} may be simplified as follows: 
\begin{quote}
(A) Let $\sX \to \sX'$ be an open immersion of modulus pairs and assume that $\sX'$ is proper.
Then, any elementary Nisnevich cover of $\sX$ can be extended to an elementary Nisnevich cover of $\sX'$.
\end{quote}

On the other hand, Theorem \ref{existence-MV} may be simplified as follows:
\begin{quote} (B) Let $N(00)$ be a \textit{proper} modulus pair. Any elementary Nisnevich square
\begin{equation}\label{intro-sq1}
\begin{CD}
\sW @>>> \sV\\
@VVV @VVV\\
\sU@>>> N(11)
\end{CD}
\end{equation}
can be embedded into an MV-square
\begin{equation}\label{intro-sq2}
\begin{CD}
N(00) @>>> N(01)\\
@VVV @VVV\\
N(10)@>>> N(11).
\end{CD}
\end{equation}
\end{quote}

The proof of Theorem \ref{th-partial} will be done in \S \ref{section-partial}.
Its basic strategy is to follow the case of a Zariski cover \cite[Lemma 4.3.5]{motmod}.
But in the Nisnevich case, we need to control the fibers of the \'etale morphisms. 
This forced us to take care of (basically) set-theoretical problems, which made the proof much more technical and complicated.
However, the statement of the theorem itself is not so surprising and easy to understand, so we do not step further into the proof here.

The proof of Theorem \ref{existence-MV} will be done in \S \ref{section-MV}, which is heart of the proof of Theorem \ref{th-presentation}.
First, it is easy to compactify a square \eqref{intro-sq1} to a square of the form \eqref{intro-sq2} such that each corner of the square is a proper modulus pair.
However, in general, such a compactification is not MV.
So, in \S \ref{subsec-6.2}, we provide a recipe to produce an MV-square out of a square of the form \eqref{intro-sq2}.
The idea is quite simple: increase the multiplicity of the boundary of the north-east corner $N(01)$ (after enlarging the total space by a sequence of blowing-ups).
After this modification, we obtain a square, denoted by $\ul{N}_1$, which satisfies the properties in Proposition \ref{eq6.9}.
Then, our task is to prove that $\ul{N}_1$ is an MV-square (Theorem \ref{MV-curve}), which will be done in the rest of the \S \ref{section-MV}.
The key point is the exactness of the sequence
\begin{multline}
\ulMCor (M,N_1(00)) \to \ulMCor(M,N_1(10)) \oplus \ulMCor(M,N_1(01))\\ \to \ulMCor(M,N_1(11)),
\end{multline}
where $M$ is any modulus pair. Take any element  $(\alpha' ,\alpha)$ in the middle term which goes to zero in $\ulMCor(M,N_1(11))$, and write $\alpha = \sum_i m_i \alpha_i$, where $\alpha_i$ are irreducible components of the algebraic cycle $\alpha$.
The proof is relatively easy if we assume that the image of the components $\alpha_i$ in $\ulMCor(M,N_1(11))$ are \textit{distinct} (see \cite[Proof of Prop. 4.3.10]{motmod} and \S \ref{pf-mv-curve} for details).
An essentially new difficulty appears if they are not distinct: for example, suppose that $\alpha = \alpha_1 - \alpha_2$, and each $\alpha_i$ goes to the same cycle $\beta$.
Then, the image of $\alpha$ is equal to zero, which implies that $\alpha'=0$.
Therefore, we lose, a priori, a way to catch the information of the boundary of  $N_1(10)$, which makes it difficult to prove that $\alpha$ comes from $\ulMCor (M,N_1(00))$.  A special case of this situation was solved in \cite[Rk. 4.3.14]{motmod}; however, we don't use ideas from loc. cit.
Rather, we prove the following surprising result, which we call ``resurgence principle'':
\begin{quote}
The assumption that two distinct cycles $\alpha_i$'s go to the same cycle $\beta$ automatically implies that $\alpha_i$'s come from $\ulMCor (M,N_1(00))$.
\end{quote}

For the precise statement, see Proposition \ref{l-resurgence}, whose proof is given in \S \ref{subsec-resurgence}.

%{\color{magenta} If this looks impossible or too technical, never mind; but I would love to insist on the fact that the deepest part of the paper is the proof of Theorem \ref{MV-curve}!}
%A key ingredient of the proof of our main result Theorem \ref{th-presentation} is the following cofinality theorem, to which this section is devoted.

Throughout the proofs, we have repeatedly used rather similar constructions; we made no serious attempt to spell them out systematically, except for Lemma \ref{l4.2}, which is used several times in \S \ref{s4.4};  even in its case, variants of it are used in \S \ref{section-partial}, but we have not tried to work out a general statement. Similarly, we use several times the ``pull-back'' of a square along a morphism to its lower right corner, and the fact that such pull-backs preserve elementary Nisnevich squares. We hope that this looseness of exposition will not disturb the reader too much.

\section{Review of modulus pairs}

We denote by $\Sch$ the category of separated $k$-schemes of finite type, and by $\Sm$ the full subcategory of smooth $k$-schemes.

According to \cite[Def. 1.1.1]{motmod}, a \emph{modulus pair} is a pair $M=(\ol{M},M^\infty)$ where $\ol{M}\in \Sch$ (the \emph{total space}), $M^\infty\subset \ol{M}$ (the \emph{boundary}) is an effective Cartier divisor, and $M^\o:=\ol{M} -M^\infty$ (the \emph{interior}) is smooth. The modulus pair $M$ is \emph{proper} if $\ol{M}$ is proper. 

\begin{remark}\label{total-reduced} By \cite[Remark 1.1.2 (3)]{motmod}, 
the total space $\ol{M}$ is reduced and $M^\o$ is dense in $\ol{M}$.
\end{remark}

Let $N$ be another modulus pair, and consider an irreducible finite correspondence $\alpha\in \Cor(M^\o,N^\o)$ as in \cite[Lect. 1]{mvw}. We say that $\alpha$ is \emph{admissible} if the following condition holds:
\begin{equation}\label{eq1.3}
M^\infty|_{\ol{\alpha}^N}\ge N^\infty|_{\ol{\alpha}^N}
\end{equation}
where $\ol{\alpha}$ is the closure of $\alpha$ in $\ol{M}\times \ol{N}$, $\ol{\alpha}^N$ is its normalization and ${-}|_{\ol{\alpha}^N}$ means ``pull-back (of Cartier divisors) to $\ol{\alpha}^N$''. 

We say that $\alpha$ is \emph{minimal} if equality holds in \eqref{eq1.3}

We also say that $\alpha$ is \emph{left proper} if the projection of $\ol{\alpha}$ to $\ol{M}$ is proper. A general finite correspondence in $\Cor(M^\o,N^\o)$ is admissible (resp. left proper) if all its components are. One shows \cite[Prop. 1.2.3]{motmod} that left proper admissible correspondences can be composed, whence an additive category $\ulMCor$; its full subcategory consisting of proper modulus pairs is denoted by $\MCor$. There is a forgetful functor
\[\omega:\MCor\to \Cor, \quad M\mapsto M^\o\]
(which extends to $\ulMCor$). We have the following important result \cite[Th. 1.6.2 and Lemma 1.11.3 (2)]{motmod}:

\begin{thm} The full embedding $\tau:\MCor\to \ulMCor$ has a pro-left adjoint $\tau^!$, given by the formula 
\[\tau^! M = ``\lim\nolimits"_{M\in \Comp_1(N)} N\]
where, for $M\in \ulMCor$, $\Comp_1(M)$ is the category whose objects are arrows $M\by{\theta} \tau(N)$, with $N\in \MCor$, such that
\begin{itemize}
\item $\theta^\o:M^\o\to N^\o$ is the identity;
\item $\theta$ defines an open immersion on the total spaces;
\item $\theta$ is minimal.
\end{itemize}
Morphisms between two objects in $\Comp_1(M)$ are given by commutative triangles.
\end{thm}

(See \cite[I.8.11.5]{SGA4} and \cite[\S A.2]{motmod} for pro-left adjoints.)

We write $\MP\subset \MCor$ and $\ulMP\subset \ulMCor$ for the subcategories with the same objects, but morphisms  restricted to (graphs of) scheme-theoretic morphisms on the interiors. Let $M,N\in \ulMP$ and $f\in \ulMP(M,N)$. We write $f\in \ulMP^\fin(M,N)$ if the rational map $\ol{M}\to \ol{N}$ defined by $f$ is a morphism; this defines a subcategory $\ulMP^\fin$ of $\ulMP$, with the same objects \cite[Def. 1.10.1]{motmod}.

We have the notion of \emph{elementary Nisnevich square} in $\ulMP^\fin$: it is a cartesian square such that all edges are minimal morphisms and the induced square on the total spaces is upper distinguished in the sense of \cite[Def. 12.5]{mvw}; see \cite[Def. 3.5.6]{motmod} for a more spelled-out definition.

The categories of modules [additive presheaves of abelian groups] over $\MCor$ and $\ulMCor$ are respectively denoted by $\MPST$ and $\ulMPST$ \cite[Def. 2.1.1]{motmod}.
We denote by $\Z_\tr$ the Yoneda embeddings $\MCor \to \MPST$ and $\ulMCor \to \ulMPST$.
%[For the definition of $\MNST$ and $\ulMNST$, could you write it? I do not know how to write in a compact way...]

Finally, one defines a notion of Nisnevich sheaves in $\ulMPST$ and $\MPST$: they form respective full subcategories $\ulMNST$ and $\MNST$ \cite[Def. 3.5.1 and 3.7.1]{motmod}. The category $\ulMP^\fin$ plays a crucial r\^ole in these definitions and, as can be expected, an elementary Nisnevich square in $\ulMP^\fin$ plays the same r\^ole as in the classical case. The functor $\Z_\tr$ takes its values in $\ulMNST$ and $\MNST$ respectively \cite[Prop. 3.5.3 and 3.7.3]{motmod}.

\section{Reduction to a cofinality statement}\label{section-presentation} 

In this section, we reduce Theorem \ref{th-presentation} to Theorem \ref{c1} below (see \S \ref{s1.2}).

\subsection{Categories of diagrams}

%\section{Application to correspondences with modulus}

Let $\sC,\Delta$ be two categories, with $\Delta$ small: we write as usual $\sC^\Delta$ for the category of functors from $\Delta$ to $\sC$. Clearly, a functor $u:\sC\to \sD$ induces a functor $u^\Delta:\sC^\Delta\to \sD^\Delta$.

In the sequel, we shall mainly consider the case where $\Delta=\Sq$: the category with 4 objects, morphisms being given by the scheme
\[\begin{CD}
00 @>>> 01\\
@VVV @VVV\\
10 @>>> 11.
\end{CD}\]

(Note that $\Sq$ is just the cartesian square of the category $[0]=\{0\to 1\}$.) 
Thus, in  \cite[\S 4]{motmod}, a square (4.1) is a certain object of $(\ulMP^\fin)^\Sq$, a square (4.3) is an object of $\MP^\Sq$, etc. 

\begin{defn}
Let $\mathcal{P}$ be a property of morphisms of the category $\ulMCor$.
Then,  a morphism $f : \ul{M} \to \ul{N}$ in the category $\ulMCor^\Sq$ has $\mathcal{P}$ if, for any $i,j \in \{0,1\}$, the morphism $f(ij) : M(ij) \to N(ij)$ has $\mathcal{P}$.
\end{defn}

\begin{example}
A morphism $f : \ul{M} \to \ul{N}$ in the category $\ulMCor^\Sq$ is minimal if all the morphisms $f(ij) : M(ij) \to N(ij)$ are minimal.
\end{example}

In the next definition, we use the comma notation $\downarrow$ of Mac Lane \cite[Ch. II, \S 6]{mcl}.

\begin{defn}\label{d1.1} For $\ul{M}\in \ulMCor^\Sq$, $\Comp_1(\ul{M})$ is the full subcategory of $\ul{M}\downarrow \MCor^\Sq$ consisting of those objects $\ul{M}\by{f} \tau^\Sq(\ul{N})$ such that $f(\delta)\in \Comp_1(M(\delta))$ for any $\delta\in Ob(\Sq)$.  
\end{defn}

We now apply the theory of Appendix \ref{sB} with $\sC=\MCor$, $\sD=\ulMCor$, $u=\tau$ and $\Delta=\Sq$.
From Lemma \ref{l2} and \cite[Lemmas 1.8.3 and 1.11.3]{motmod}, we get:

\begin{prop}\label{p1} The functor $\tau^\Sq:\MCor^\Sq\to \ulMCor^\Sq$ has a pro-left adjoint, which is represented by 
$\ul{M}\mapsto \Comp_1(\ul{M})$. 
\end{prop}

\begin{proof} By definition, $\Comp_1(M)$ is full in $M\downarrow \tau$ for any $M\in \ulMCor$. By Lemma \ref{l2} b), $\Comp_1(\ul{M})$ coincides with the category $I(\ul{M})$ introduced there, and Proposition \ref{p1} follows from part c) of this lemma.
%$M\by{f} \tau(N)$ is a monomorphism for any $M\in \ulMCor$ and any $f\in \Comp_1(M)$, which in turn follows from the faithfulness of $\ulomega:\ulMCor\to \Cor$ and the fact that $\ulomega(f)=1_{\ulomega(M)}$.
\end{proof}

For any additive category $\sA$, write $C(\sA)$ for the category of chain complexes on $\sA$. Let $C\in C(\MPST)$. Viewing $C$ as a functor $\MCor^\op\to C(\Ab)$, we get the functor $C^\Sq:(\MCor^\op)^\Sq\to C(\Ab)^\Sq$. For $\ul{M}\in \ulMCor^\Sq$, we therefore get the complex 
\[D=\Tot \tau_!^\Sq C^\Sq(\ul{M})\in C(\Ab).\]

\begin{cor} \label{co1} The natural map
\[\colim_{\ul{N}\in \Comp_1(\ul{M})} \Tot C^\Sq(\ul{N})\to D\]
is an isomorphism.
\end{cor}

\begin{proof} First a word on the ``natural map'':  For any $\ul{M}\to \tau^\Sq \ul{N}\in \Comp_1(\ul{M})$, applying the functor $\tau_!^\Sq C^\Sq:(\ulMCor^\op)^\Sq\to C(\Ab)^\Sq$ we get a composite map in $C(\Ab)^\Sq$
\[C^\Sq(\ul{N})\to\tau_!^\Sq C(\tau^\Sq(\ul{N}))\to \tau_!^\Sq C(\ul{M})\]
where the first map is given by Yoneda. By Proposition \ref{p1}, the induced map
\[\colim_{\ul{N}\in \Comp_1(\ul{M})} C^\Sq(\ul{N})\to \tau_!^\Sq C(\ul{M})\]
is an isomorphism, and the conclusion follows from the fact that $\Tot$ commutes with filtering colimits.
\end{proof}

For the sequel, we need the following definition and (trivial) lemma.

\begin{defn} \label{d1} A morphism $f:M\to M'$ in $\ulMCor$ is an \emph{open immersion} if $f\in \ulMP^\fin$ \cite[Def. 1.10.1]{motmod}, $f$ is minimal (ibid., Def. 1.10.2 a)) and $\ol{f}:\ol{M}\to \ol{M'}$ is an open immersion (cf. \cite[Def. 3.1.1]{motmod}).
\end{defn}

\begin{lemma}\label{l4} Let $f:\ul{M}\to \ul{N}$ be a morphism in $\ulMCor^\Sq$. Suppose that, for any $\delta\in \Sq$, $f(\delta)$ is a open immersion in the sense of Definition \ref{d1} and $\ulomega(f(\delta))=1_{M^\o(\delta)}$. (In other words, $f(\delta)$ verifies the conditions of \cite[Lemma 1.11.3 (1)]{motmod}, except for the properness of $N(\delta)$). Then $f$ induces a functor $f^*:\Comp_1(\ul{N})\to \Comp_1(\ul{M})$.\qed
\end{lemma}

\begin{comment}
{\color{blue}
We shall also use the following construction and lemma many times.

\begin{defn}\label{d2.2} a) Let $M=(\ol{M},M^\infty)\in \ulMCor$, and let $f:\ol{M'}\to \ol{M}$ be a dominant morphism in $\Sch$ such that $f^{-1}(M^\o)$ is smooth. We define a modulus pair $f^*M\in \ulMCor$ by
\[f^*M=(\ol{M'},f^*M^\infty).\]
It is provided with a minimal morphism $f:f^*M\to M$.\\
b) Let $\ul{M}\in (\ulMP^\fin)^\Sq$, and let $f:\ol{M'}\to \ol{M(11)}$ be a dominant morphism in $\Sch$ such that $f^{-1}(M^\o)$ is smooth. For all $i,j\in \{0,1\}$, let $\ol{M'}(ij)=\ol{M'}\times_{\ol{M(11)}}\ol{M(ij)}$, and let $f(ij):\ol{M'}(ij)\to\ol{M(ij)}$ be the projection. We define a square $f^*\ul{M}\in (\ulMP^\fin)^\Sq$ by
\[(f^*\ul{M})(ij)=f^*M(ij).\]
It is provided with a minimal morphism $f:f^*\ul{M}\to \ul{M}$.
\end{defn}

}
\end{comment}

\subsection{Cofinality of MV-squares}\label{s1.2}

\begin{defn}\label{d2.3} a) A square $\ul{N}\in \MCor^\Sq$ is \emph{MV} if 
\begin{thlist}
\item $\omega^\Sq(\ul{N})\in \Cor^\Sq$ is an upper distinguished (=elementary Nisnevich) square; %{\color{red} [What is the difference between ``distinguished'' and ``elementary''? Is it just a problem of terminology?]}
\item the complex $\Tot \Z_\tr^\Sq(\ul{N})$ \cite[(4.4)]{motmod} is exact in $\MNST$. 
\end{thlist}
b) Consider an elementary Nisnevich square
\begin{equation}\label{eq3a} 
\ul{S}\, =\quad
\begin{CD}
S(00)@>a>> S(01)\\
@VcVV @Vd VV\\
S(10)@>b>> S(11)
\end{CD}
\end{equation}
as in  \cite[(4.1)]{motmod}. An object $\ul{S}\to \tau^\Sq\ul{N}\in \Comp_1(\ul{S})$ is an \emph{MV completion} of $\ul{S}$ if $\ul{N}$ is MV. 
We write $\Comp_1^{MV}(\ul{S})$ for the full subcategory of $\Comp_1(\ul{S})$ consisting of MV completions.
\end{defn}

\begin{rk} In Definition \ref{d2.3} a), Condition (i) was put for philosophical reasons and will not be used in the proofs. This condition is automatic for MV completions in the sense of Definition \ref{d2.3} b), and all MV squares which will appear in this paper are of this form. 
\end{rk}

% We shall prove in the following sections

\begin{thm}\label{c1}  If $\ol{S}(11)$ is normal, the category $\Comp_1^{MV}(\ul{S})$ is cofinal in $\Comp_1(\ul{S})$.
\end{thm}

%The aim of this section is to show:

\begin{prop}\label{p2.2} Theorem \ref{th-presentation} follows from Theorem \ref{c1}.
\end{prop}

\begin{proof} Consider the square in Theorem \ref{th-presentation}.
By \cite[Ex. A.11.6 and Th. A.11.9]{motmod}, $\alpha$ is fully faithful; by loc. cit., Prop. 6.9.1 (2), $\gamma$ is essentially surjective, and $\beta$ is a (full) embedding by definition (loc. cit., Def. 6.9.3); therefore it suffices to show that $\delta$ is an equivalence of categories. We argue similarly to the proof of \cite[Prop. 6.3.2]{motmod}.  First, $D(a_\Nis):D(\MPST)\to D(\MNST)$ is a localisation functor by the same argument (referring to prop. 3.7.3 rather than prop. 3.5.3). Next, $\Z_\tr(MV)$ is a set of compact objects of $D(\MPST)$, contained in $\Ker D(a_\Nis)$ by definition of $MV$ (see \cite[Lemmas 4.3.2, 4.3.3 and Def. 4.3.7]{motmod}). It suffices to show that $\Ker D(a_\Nis)$ is generated by $\Z_\tr(MV)$, or equivalently by \cite[Th. A.11.7]{motmod}, that the right orthogonal of these objects in $\Ker D(a_\Nis)$ is $0$. Consider the naturally commutative diagram
\[\begin{CD}
\Ker D(a_\Nis) @>>> \Ker D(\ul{a}_\Nis)\\
@VVV @VVV\\
D(\MPST)@>D(\tau_!)>> D(\ulMPST)\\
@VD(a_\Nis)VV @VD(\ul{a}_\Nis) VV\\
D(\MNST)@>D(\tau_\Nis)>> D(\ulMNST).
\end{CD}\]

Note that $D(\tau_!)$ sends $\Z_\tr(MV)$ to compact objects of $D(\ulMPST)$. Since this functor is fully faithful, it suffices to show:
\end{proof}

\begin{lemma} Let $C\in D(\MPST)$. If $C$ is right orthogonal to $\Z_\tr(MV)$, then $D(\tau_!) C$ is right orthogonal to $\Z_\tr(\ul{MV})$, where $\ul{MV}$ are the MV relations in $\ulMPST$ (denoted by (MV2) in \cite[\S 6.3]{motmod}).
\end{lemma}

\begin{proof} Consider an elementary Nisnevich square $\ul{S}$ \eqref{eq3a}. Since $\Z_\tr(S(00))$, \dots, $\Z_\tr(S(11))$ are projective objects of $\ulMCor$, we have to show that the complex of abelian groups
\begin{multline}\label{eq2}
\tau_! C(S(00))\to \tau_! C(S(01))\oplus \tau_! C(S(10))\to \tau_! C(S(11))\\ = \Tot \tau_!^\Sq C^\Sq(\ul{S})
\end{multline}
is acyclic. 

\begin{lemma}\label{l-normal} We may assume that $\ol{S}(11)$ is normal. 
\end{lemma}

\begin{proof} 
%{\color{red} {\large I have spelled it out. Is it fine?}
For any $i,j=0,1$, let $\pi_{ij} : \ol{S} (ij)^N \to \ol{S} (ij)$ be the normalization and let $S(ij)^N=(\ol{S} (ij)^N,\pi_{ij}^*S(ij)^\infty)$. 
Then, $\pi_{ij}$ defines an isomorphism of modulus pairs $S(ij)^N\iso S(ij)$ in $\ulMCor$ for each $i,j$ since it is minimal, proper  and induces an isomorphism $(S(ij)^N )^\circ \iso S(ij)^\circ$.
Since the maps $\ol{S}(ij) \to \ol{S}(11)$ are \'etale for all $i,j$, we have $\ol{S} (ij)^N = \ol{S} (ij) \times_{\ol{S}(11)} \ol{S}(11)^N$.
Therefore the $S(ij)^N$'s form an elementary Nisnevich square $\ul{S}^N$, and the natural morphism $\ul{S}^N \to \ul{S}$, induced by $\pi_{ij}$'s, is an isomorphism of squares.
Therefore, we may replace everything by the pullbacks along $\pi$.
\end{proof}

Assume $\ol{S}(11)$ normal.
Then Corollary \ref{co1} and Theorem \ref{c1} yield an isomorphism
\[\colim_{\ul{N}\in \Comp_1^{MV}(\ul{S})}  C^\Sq(\ul{N})\iso  \tau_!^\Sq C(\ul{S}).\]
The acyclicity of \eqref{eq2} follows, since $\Tot$ obviously commutes with colimits and the complex
\[\Tot C^\Sq(\ul{N}) =C(N(00))\to C(N(01))\oplus C(N(10))\to C(N(11))\]
is acyclic for any $\ul{N}\in \Comp_1^{MV}(\ul{S})$ by hypothesis.
\end{proof}

\begin{rk}\label{r2.1} if $\ol{S}(11)$ is normal, so is $\ol{S}(ij)$ for all $(i,j)$ since the morphisms $\ol{S}(ij)\to \ol{S}(11)$ are \'etale.
\end{rk}

\begin{rk} Theorem \ref{c1} says that there are ``enough'' MV squares in $\MCor^\Sq$.  A sufficient condition for a square to be MV is given in Theorem \ref{MV-curve}. This raises the question: can one describe all MV squares? \end{rk}

\section{Further reductions}

In this section, we reduce Theorem \ref{c1} to the following two statements. 
For the first one, we need a definition:

\begin{defn}\label{d2.1}\
\begin{enumerate}
\item
A morphism $\alpha : M \to N$ in $\ulMCor$ is called 
\begin{itemize}
\item \emph{entire} if $\alpha \in \Cor (M^\circ , N^\circ )$ is a morphism of schemes $M^\circ \to N^\circ$, and extends to a morphism of schemes $f_\alpha : \ol{M} \to \ol{N}$;
\item \emph{minimal} if $\alpha$ is entire and $M^\infty = f_\alpha^\ast N^\infty$;
\item an \emph{extension} if $\alpha$ is minimal, $f_\alpha$ is an open immersion, and induces an isomorphism $M^\circ \iso N^\circ$.
\end{itemize}
\item 
Let $\ul{S}\in \ulMCor^\Sq$.
An \emph{extension} of $\ul{S}$ is a morphism of squares $f:\ul{S}\to \ul{S}'$ such that $f(ij)$ is an extension for all $(i,j)$. An extension of $\ul{S}$ is \emph{strict} if,  for all $i,j\in \{0,1\}$, we have $\ol{S}(ij) = \ol{S}'(ij) \times_{\ol{S}'(11)} \ol{S}(11)$.
\end{enumerate}
\end{defn}

Note that, in particular, any object of $\Comp_1(\ul{S})$ defines an extension of $\ul{S}$.

\begin{thm}\label{th-partial}
Let $\ul{S}$ be an elementary Nisnevich square, with $\ol{S}(11)$ normal.
Let $\ul{S}\by{a} \ul{N}_0 \in \Comp_1 (\ul{S})$ be a compactification of $\ul{S}$.
Then, there exists a commutative diagram in $\ulMCor^\Sq$
\begin{equation}\label{eq2.1}
\begin{CD}
\ul{S}'@>c>> \ul{N}_1\\
@Ab AA @Vd VV\\
\ul{S} @>a>> \ul{N}_0
\end{CD}
\end{equation}
such that 
\begin{thlist}
\item $\ul{S}'$ is an elementary Nisnevich squanre;
\item $\ol{S'}(11)$ is proper and normal;
\item $b$ is a strict extension;
\item $cb:\ul{S}\to \ul{N}_1$ is an object of $\Comp_1(\ul{S})$;
\item $c$ is minimal.
\end{thlist}
\end{thm}

The second one is:

\begin{thm}\label{existence-MV}
Theorem \ref{c1} is true when $\ol{S}(11)$ is proper.
\end{thm}

\begin{prop}\label{p2.1} Theorem  \ref{th-partial} and Theorem \ref{existence-MV} imply Theorem \ref{c1}.
\end{prop}

\begin{proof} We start from $a:\ul{S}\to \ul{N}_0$ in $\Comp_1(\ul{S})$, and give ourselves a commutative diagram \eqref{eq2.1} 
satisfying the conditions of Theorem \ref{th-partial}.   Choose $[\ul{S}'\by{g} \ul{N}_2] \in \Comp_1(\ul{S}')$. Define a square of schemes $\ol{N}_3$ as follows: $\ol{N}_3(ij)$ is the scheme-theoretic image of $\ol{S}'(ij)$ in $\ol{N}_1(ij)\times \ol{N}_2(ij)$ (a closed subscheme, \cite[Ch. I, Def. 6.10.1]{EGA1}). 
By \cite[Ch. I, Prop. 6.10.5]{EGA1}, it exists, and the corresponding morphism $f(ij) : \ol{S}'(ij) \to \ol{N_3}(ij)$ is scheme-theoretically dominant.
Write $p_\epsilon :\ol{N}_3\to \ol{N}_\epsilon$ ($\epsilon = 1,2$). %$p_2:\ol{N}_3\to \ol{N}_2$ 
for the projection. Applying Lemma \ref{no-extra-fiber} to the diagram
\[\begin{xymatrix}{
 & \ol{N}_3(ij) \ar[d]^{p_2(ij)} \\
 \ol{S}'(ij)\ar[ru]^{f(ij)} \ar[r]^{g(ij)}& \ol{N}_2(ij)
}\end{xymatrix}\]
we see that 
$f(ij)$ is an open immersion because $g(ij)$ is. 
Defining
\[N_3(ij)^\infty = p_1(ij)^* N_1(ij)^\infty\] 
we get $[\ul{S}'\to \ul{N}_3] \in \Comp_1(\ul{S}')$, thanks to (v).

By Theorem \ref{existence-MV}, choose $[\ul{S}'\to \ul{N}_4]\in \Comp_1^{MV}(\ul{S}')$ mapping to $[\ul{S}'\to \ul{N}_3]$. By Lemma \ref{l4}, $[\ul{S}\to \ul{N}_4]\in \Comp_1^{MV}(\ul{S})$ and maps to $[\ul{S}\to \ul{N}_1]$, hence to $[\ul{S}\to \ul{N}_0]$.
\end{proof}

\begin{remark}
In  \cite[Ch. I, Prop. 6.10.5]{EGA1}, there are two sufficient conditions for the existence of $\ol{N}_3(ij)$; (1) the morphism $\ol{S}'(ij) \to \ol{N}_1(ij)\times \ol{N}_2(ij)$ is quasi-compact and quasi-separated, or (2) $\ol{S}'(ij)$ is reduced.
Both are true, the second one by Remark \ref{total-reduced}.
\end{remark}

\section{Cofinality of universally minimal squares}

The aim of this section is to prove Proposition \ref{l-cofinality-minimal} below. It will be used in \S \ref{section-MV}.

\subsection{Universally minimal squares}
\begin{defn}[cf. \protect{\cite[Def. 4.3.9]{motmod}}]\label{defn-minimal-sq} %Let $\ul{S}$ be an elementary Nisnevich square in $\ulMCor$.
A square $\ul{N} \in \MCor^\Sq$ whose morphisms are entire (Definition \ref{d2.1}) is called \emph{universally minimal} if the following condition holds: write
\begin{equation}\label{name}
\begin{CD}
\ol{N}(00)@>h_u>> \ol{N}(01)\\
@Vv_lVV @Vv_r VV\\
\ol{N}(10)@>h_d>> \ol{N}(11).
\end{CD}
\end{equation}
Then, for any normal $k$-scheme $Y$ and for any $k$-morphism $f : Y \to \ol{N}(00)$ such that the pullback $f^\ast N(00)^\infty$ is well-defined, we have 
\[
f^\ast N(00)^\infty = \mathrm{sup} (f^\ast h_u^\ast N(01)^\infty , f^\ast v_l^\ast N(10)^\infty ),
\]
where $\mathrm{sup}$ is taken as Weil-divisors.
Note that, since $|h_u^\ast N(01)^\infty| \cup |v_l^\ast N(10)^\infty| \subset |N(00)^\infty|$, the pullbacks appearing on the right hand side of the equality are automatically well-defined.
Write $\Comp^{\min}_1(\ul{S})$ for the full subcategory of $\Comp_1^{\mathrm{mor}}(\ul{S})$ consisting of those objects $X:\ul{S}\to \tau^\Sq\ul{N}$ such that $\ul{N}$ is universally minimal.
\end{defn}

\begin{prop}[Cofinality of minimal squares]\label{l-cofinality-minimal}
Let $\ul{S}$ be an elementary Nisnevich square, and assume that $\ol{S}(11)$ is normal.
Then, $\Comp^{\min}_1(\ul{S})$ is cofinal in $\Comp_1(\ul{S})$.
\end{prop}

\subsection{Three elementary lemmas}

\begin{lemma}\label{l4.0} Let $\ulMCor^{\Sq,\mor}$ be the full subcategory of $\ulMCor^\Sq$ formed of those squares $\ul{N}$ in which all morphisms are given by entire morphisms (Definition \ref{d2.1}). Then the full embedding 
\[\ulMCor^{\Sq,\mor}\inj\ulMCor^\Sq\] 
is an equivalence of categories. 
\end{lemma}

\begin{proof} 
Let $\ul{N}\in \ulMCor^\Sq$. Using Notation \eqref{name}, define a new square of schemes $\ul{\ol{N'}}$ as follows: 
\begin{align*}
\ol{N'}(11) &= \ol{N(11)};\\
\ol{N'}(10) &= \Gamma_{h_d};\\
\ol{N'}(01) &= \Gamma_{v_r};\\
\ol{N'}(00) &= \Gamma_{\phi}
\end{align*}
where $\Gamma_f$ denotes the [closure of the] graph of a rational map $f$, and $\phi$ is the rational map $\ol{N}(00)\tto \Gamma_{h_d}\times  \Gamma_{v_r}$ determined by $h_u$ and $v_l$. Then the rational maps of  \eqref{name} ``resolve'' to morphisms in $\ul{\ol{N'}}$, and we have a canonical morphism $\pi(ij):\ol{N'}(ij)\to \ol{N(ij)}$ for all $(i,j)$. Define
\[N'(ij)^\infty = \pi(ij)^*N(ij)^\infty.\]

This defines modulus pairs $N'(ij)$ and isomorphisms $\pi(ij):N'(ij)\iso N(ij)$ in $\ulMP$. The latter fact implies that the morphisms of $\ul{\ol{N'}}$ are admissible, whence an object $\ulMCor^{\Sq,\mor}$ together with an isomorphism $\pi:\ul{N'}\iso \ul{N}$ in $\Comp_1(\ul{S})$.
\end{proof}

\begin{rk} This lemma, and its proof, extend to the case where $\Sq$ is replaced by any finite category without loops.
\end{rk}

\begin{lemma}\label{l4.1} Let $\ul{S}$ be an elementary Nisnevich square, and let $\Comp_1^\mor(\ul{S})$ be the full subcategory of $\Comp_1 (\ul{S})$ formed of those squares $\ul{N}$ in which all morphisms are given by entire morphisms (Definition \ref{d2.1}). Then the full embedding $\Comp_1^\mor(\ul{S})\inj\Comp_1(\ul{S})$ is an equivalence of categories. In particular,  $\Comp_1^\mor(\ul{S})$ is cofinal in $\Comp_1(\ul{S})$.
\end{lemma}

\begin{proof} This follows immediately from Lemma \ref{l4.0}.
\end{proof}

\begin{lemma}\label{junk-separation}
Let $\ul{S}$ be an elementary Nisnevich square, and let $\ul{N} \in \Comp_1^\mor (\ul{S})$.
%Assume that all morphisms in the square $\ul{N}$ are given by morphisms of schemes on the total spaces. 
Name those morphisms as in the diagram \eqref{name} of Definition \ref{l-cofinality-minimal}.
Then, we have
\[
h_u^{-1} (\ol{S}(01)) \cap v_l^{-1} (\ol{S}(10)) = \ol{S}(00).
\]
\end{lemma}

\begin{proof}
By the universality of the fiber product, we obtain a unique morphism $f : \ol{N}(00) \to \ol{N}(01) \times_{\ol{N}(11)} \ol{N}(10)$ such that $h_u$ and $v_l$ factor through $f$.
The restriction of $f$ to $\ol{S}(00)$ factors through the map
\[\ol{S}(00) \to \ol{S}(01) \times_{\ol{N}(11)} \ol{S}(10) = \ol{S}(01) \times_{\ol{S}(11)} \ol{S}(10)\]
%\[\ol{S}(00) \twoheadrightarrow f(\ol{S}(00)) \subset \ol{S}(01) \times_{\ol{N}(11)} \ol{S}(10) = \ol{S}(01) \times_{\ol{S}(11)} \ol{S}(10) \cong \ol{S}(00),\]
which is an isomorphism by the hypothesis on $\ul{S}$. %This implies that 
%\[f(\ol{S}(00)) = \ol{S}(01) \times_{\ol{S}(11)} \ol{S}(10),\]
%and the morphism $\ol{S}(00) \twoheadrightarrow f(\ol{S}(00))$ is an isomorphism.
%Therefore, s
Since $f$ is separated, Lemma \ref{no-extra-fiber} then implies that 
\begin{equation}\label{eq002}
f^{-1} (\ol{S}(01) \times_{\ol{N}(11)} \ol{S}(10)) = f^{-1} (f(\ol{S}(00))) = \ol{S}(00).
\end{equation}
Since we have
\[
\ol{S}(01) \times_{\ol{N}(11)} \ol{S}(10) = (\ol{S}(01) \times_{\ol{N}(11)} \ol{N}(10)) \cap (\ol{N}(01) \times_{\ol{N}(11)} \ol{S}(10)),
\]
the left hand side of (\ref{eq002}) can be rewritten as 
\begin{align*}
&f^{-1} (\ol{S}(01) \times_{\ol{N}(11)} \ol{S}(01)) \\
&= f^{-1} ( (\ol{S}(01) \times_{\ol{N}(11)} \ol{N}(10)) ) \cap f^{-1} ( (\ol{N}(01) \times_{\ol{N}(11)} \ol{S}(10)) ) \\
&=h_u^{-1} (\ol{S}(01)) \cap v_l^{-1} (\ol{S}(10)).
\end{align*}
This finishes the proof of Lemma \ref{junk-separation}.
\end{proof}

\subsection{A useful construction}\label{s4.3}

\begin{lemma}\label{l4.2}
Let $\ul{S}$ be an elementary Nisnevich square such that $\ol{S}(11)$ is normal, and take any $\ul{N}_0 \in \Comp_1(\ul{S})$. 
Let $W \subset \ol{N}_0(00)$ be a closed subscheme whose restriction $W \times_{\ol{N}_0(00)} \ol{S}(00)$ to the open subset $\ol{S}(00)$ is an effective Cartier divisor on $\ol{S}(00)$. 
Consider the modulus pair
\[
N_0'(00)=(\mathrm{Bl}_{W} (\ol{N}_0(00))^N, \text{ the pullback of } N_0(00)^\infty).
\]
Then the projection $\pi:N_0'(00)\to N_0(00)$ is an isomorphism, and the morphism $S(00)\to N_0(00)$ lifts to a morphism $S(00)\to N$ in $\ulMP^\mor$. Moreover, $\pi^{-1}(W)$ is an effective Cartier divisor on $\ol{N}_0'(00)$, with same restriction to $\ol{S}(00)$ as $W$.
\end{lemma}

\begin{proof} The projection $\mathrm{Bl}_{W} (\ol{N}_0(00))^N \to \mathrm{Bl}_{W} (\ol{N}_0(00)) \to \ol{N}_0(00)$ is an isomorphism over $\ol{S}(00)$ by the universal property of blow-up and by the assumption that $\ol{S}(00)$ is normal. The last claim is obvious.
\end{proof}

\begin{rk} \label{r4.1} Let $W_1,\dots,W_r$ be closed subschemes of $\ol{N}_0(00)$ verifying the hypothesis of Lemma \ref{l4.2}. Applying its construction repeatedly, we obtain a modulus pair $N_0^{(r)}(00)$ with $\ol{N}^{(r)}_0(00)$ normal and a morphism $\pi_r:N_0^{(r)}(00)\to N_0(00)$ in $\ulMP^\mor$ which is an isomorphism in $\ulMP$, such that $S(00)\to N_0(00)$ lifts to $N_0^{(r)}(00)$ and, for any $i$, $\pi_r^{-1}(W_i)$ is a Cartier divisor with same restriction to $\ol{S}(00)$ as $W_i$. Of course, $N_0^{(r)}(00)$ depends in general on the ordering on the $W_i$'s.
\end{rk}

%Therefore, we may assume that $W$ is an effective Cartier divisor, and that $\ol{N}_0(00)$ is normal.
In the next subsection, we will make use of Lemma \ref{l4.2} and Remark \ref{r4.1} several times for various $W$'s.

\subsection{Proof of Proposition \ref{l-cofinality-minimal}} \label{s4.4}
Take any $\ul{N}_0 \in \Comp_1 (\ul{S})$.
By Lemma \ref{l4.1}, we may assume that $\ul{N}_0 \in \Comp_1^{\mathrm{mor}} (\ul{S})$.
%Before beginning the proof, 

%Now, we begin the proof. 
For each $i,j \in \{0,1\}$, define $D_{ij} := \ol{N}_0(ij) \setminus \ol{S}(ij)$, and regard $D_{ij}$ as a closed subscheme of $\ol{N}_0(ij)$ with the reduced scheme structure. 
Recall the notation in Diagram \eqref{name}, and set \[p:=v_r \circ h_u = h_d \circ v_l.\]
%For each $i,j \in \{0,1\}$, we set 
We set
\[
\tilde{N}_0(00)^\infty := N_0(00)^\infty - p^\ast N_0(11)^\infty .
\]
By the admissibility of $p$, the Cartier divisor $\tilde{N}_0(00)^\infty$ is effective.
Since the restrictions of $N_0(00)^\infty$ and $p^\ast N_0(11)^\infty$ to $\ol{S}(00)$ both equal $S(00)^\infty$, we have
\[\tilde{N}_0(00)^\infty \times_{\ol{N}_0(00)} \ol{S}(00) = 0,\] 
hence
\[
|\tilde{N}_0(00)^\infty | \subset \ol{N}_0(00) \setminus \ol{S}(00) = |D_{00}|.
\]
Moreover, by applying \S \ref{s4.3}, we may reduce to the case where the closed subschemes $v_l^\ast D_{10}$, $h_u^\ast D_{01}$ and $v_l^\ast D_{10} \times_{\ol{N}_0(00)} h_u^\ast D_{01}$ are effective Cartier divisors on $\ol{N}_0(00)$, because their restrictions to $\ol{S}(00)$ are empty.
Then, by Lemma \ref{junk-separation}, we have
\begin{align*}
|D_{00}| &= \ol{N}_0(00) \setminus \ol{S}(00) = \ol{N}_0(00) \setminus (h_u^{-1} (\ol{S}(01)) \cap v_l^{-1} (\ol{S}(10))) \\
&=(\ol{N}_0(00) \setminus h_u^{-1} (\ol{S}(01))) \cup (\ol{N}_0(00) \setminus v_l^{-1} (\ol{S}(10))) \\
&= |h_u^\ast D_{01}| \cup |v_l^\ast D_{10}| \\
&= |\mathrm{sup}(h_u^\ast D_{01},v_l^\ast D_{10})|,
\end{align*}
where the sup is taken as Weil divisors, but it is also the sup as Cartier divisors by Lemma \ref{sup-lemma-pre} and by the assumption that $v_l^\ast D_{10} \times_{\ol{N}_0(00)} h_u^\ast D_{01}$ is an effective Cartier divisor.
Therefore, by Lemma \ref{m-i-l}, there exists a positive integer $m>0$ such that
\[
\tilde{N}_0(00)^\infty \leq m\mathrm{sup}(v_l^\ast D_{10},h_u^\ast D_{01})=\mathrm{sup}(mv_l^\ast D_{10},mh_u^\ast D_{01}).
\]
Again by \S \ref{s4.3}, we may assume that the closed subscheme 
\[(p^\ast N_0(11)^\infty + mv_l^\ast D_{10}) \times_{\ol{N}_0(00)} (p^\ast N_0(11)^\infty +mh_u^\ast D_{01})\] 
is an effective Cartier divisor on $\ol{N}_0(00)$, because its restriction to $\ol{S}(00)$ is the effective Cartier divisor $S(00)^\infty$.
Since $\tilde{N}_0(00)^\infty = N_0(00)^\infty - p^\ast N_0(11)^\infty$ by definition, we obtain
\[
N_0(00)^\infty \leq \mathrm{sup}(p^\ast N_0(11)^\infty + mv_l^\ast D_{10},p^\ast N_0(11)^\infty +mh_u^\ast D_{01}),
\]
where the sup is taken as Weil divisors, but it is also the sup as Cartier divisors by Lemma \ref{sup-lemma-pre}, and by the assumption that $(p^\ast N_0(11)^\infty + mv_l^\ast D_{10}) \times_{\ol{N}_0(00)} (p^\ast N_0(11)^\infty +mh_u^\ast D_{01})$ is an effective Cartier divisor.
By $p=v_r \circ h_u = h_d \circ v_l$, and by the admissibility of $h_d$ and $v_r$, we have
\begin{align*}
p^\ast N_0(11)^\infty &= v_l^\ast h_d^\ast N_0(11)^\infty \leq v_l^\ast N_0(10)^\infty ,\\
p^\ast N_0(11)^\infty &= h_u^\ast v_r^\ast N_0(11)^\infty \leq h_u^\ast N_0(01)^\infty .
\end{align*}

Still by \S \ref{s4.3}, we may assume that the closed subscheme $(v_l^\ast (N_0(10)^\infty + mD_{10})) \times_{\ol{N}_0(00)} (h_u^\ast (N_0(01)^\infty + mD_{01}))$ is an effective Cartier divisor on $\ol{N}_0(00)$, because its restriction to $\ol{S}(00)$ is the effective Cartier divisor $S(00)^\infty$.
Combining the inequalities above, we obtain
\begin{equation}\label{eq003}
N_0(00)^\infty \leq \mathrm{sup}(v_l^\ast (N_0(10)^\infty + mD_{10}), h_u^\ast (N_0(01)^\infty + mD_{01})),
\end{equation}
where the sup is a sup of Cartier divisors, as previously.

\ 

Define modulus pairs $N_1(ij)$, $i,j\in \{0,1\}$, by 
\begin{align*}
N_1(11) &:= N_0(11) = (\ol{N}_0(11) , N_0(11)^\infty ), \\
N_1(10) &:= (\ol{N}_0(10) , N_0(10)^\infty + mD_{10}), \\
N_1(01) &:= (\ol{N}_0(01) , N_0(01)^\infty + mD_{01}), \text{ and }\\
N_1(00) &:= (\ol{N}_0 (00) , \sup \{v_l^\ast(N_0(10)^\infty + mD_{10}) ,h_u^\ast (N_0(01)^\infty + mD_{01}) \}).
\end{align*}

First, we check that these modulus pairs form a square.
The admissibility of the morphisms $N_0(10) \to N_0(11)$ and $N_0(01) \to N_0(11)$ imply that we have admissible morphisms $N_1(10) \to N_1(11)$ and $N_1(01) \to N_1(11)$.
Moreover, since $N_1(00)^\infty \geq h_u^\ast N_1(01)^\infty$ and $N_1(00)^\infty \geq v_l^\ast N_1 (10)^\infty$ by definition, we have admissible morphisms $N_1(00) \to N_1(10)$ and $N_1(00) \to N_1(01)$.
Therefore, we obtain a square $\ul{N}_1 \in \MCor^\Sq$.

Next, we check that $\ul{N}_1$ belongs to $\Comp_1^\mathrm{min} (\ul{S})$.
Since $N_1(11)^\infty = N_0(11)^\infty$ by definition, we have $N_1(11)^\infty |_{\ol{S}(11)} = S(11)^\infty$ by the minimality of $S(11) \to N_0(11)$.
Next, let $(i,j)$ be either $(1,0)$ or $(0,1)$.
Since $D_{ij} \cap \ol{S}(ij) = \emptyset$ by the definition of $D_{ij}$, we have $N_1(ij)^\infty |_{\ol{S}(ij)} = N_0(ij)^\infty |_{\ol{S}(ij)} = S(ij)^\infty$. 
Finally, we have
\begin{align*}
N_1(00)^\infty |_{\ol{S}(00)} &= \sup \{h_u^\ast N_1(01)^\infty ,v_l^\ast N_1 (10)^\infty \} |_{\ol{S}(00)} \\
&=\sup \{N_1(01)^\infty |_{\ol{S}(00)} , N_1 (10)^\infty |_{\ol{S}(00)} \} \\
&=\sup \{S(00)^\infty  , S(00)^\infty  \} \\
&=S(00)^\infty .
\end{align*}
Therefore, we have $\ul{N}_1 \in \Comp_1 (\ul{S})$.
The universal minimality of $\ul{N}_1$ holds by the definition of $\ul{N}_1$, by the assumption that $(v_l^\ast (N_0(10)^\infty + mD_{10})) \times_{\ol{N}_0(00)} (h_u^\ast (N_0(01)^\infty + mD_{01}))$ is an effective Cartier divisor, and by Lemma \ref{sup-lemma-pre}.

Finally, we check that $\ul{N}_1$ dominates $\ul{N}_0$.
For any $i,j$, noting that $\ol{N}_1(ij) = \ol{N}_0(ij)$, the identity morphism on total spaces induces an admissible morphism $N_1(ij) \to N_0(ij)$ since $N_1(ij)^\infty \geq N_0(ij)^\infty$ by definition if $(i,j) \neq (0,0)$ and by (\ref{eq003}) if $(i,j)=(0,0)$.
This finishes the proof of Proposition \ref{l-cofinality-minimal}.

\section{Existence of partial compactifications}\label{section-partial}

This section is devoted to the proof of Theorem \ref{th-partial}.

\subsection{A first reduction}\label{main-partial}

In this subsection, we prove the following result:

\begin{prop}\label{th-partial-prepre}
Let $\ul{S}$ be an elementary Nisnevich square, and take any $S'(11) \in \Comp_1 (S(11))$.
Then, there exists a cartesian diagram of schemes
\[\begin{xymatrix}{
\ol{S}(01) \ar[r] \ar[d]^f \ar@{}[rd]|\square& V \ar[d]^{q_V}  \\
\ol{S}(11) \ar[r]^g & \ol{S}'(11) ,
}\end{xymatrix}\]
satisfying the following properties:
\begin{enumerate}
%\item[(1)] $\ol{S}'(11)$ is proper.
\item The horizontal arrows are open immersions.
\item The map $f$ (resp. $g$) is the (underlying) structure map of the square $\ul{S}$ (resp. of the morphism $S'(11)\to S'(11)$).
%\item There exists a $k$-morphism $\ol{S}'(11) \to \ol{N}_0(11)$ which restricts to the identity on $\ol{S}(11)$.
\item Set $Z(00) := \ol{S}(01) \setminus \ol{S}(00)$. Then, the closure $\ol{Z(00)}$ of $Z(00)$ in $V$ is proper over $\ol{S}'(11)$.
\item The morphism $q_V$ is flat, and $V$ is separated.
%Define $Z(10) := \ol{S}(11) \setminus \ol{S}(10)$ and $Z(00) := \ol{S}(01) \setminus \ol{S}(00)$.
%Denote by $\ol{Z(10)}$ and $\ol{Z(00)}$ the closures in $\ol{S}'(11)$ and $V(01)$, respectively.
%Then, the morphism $\ol{Z(10)}_\red \to \ol{Z(00)}_\red$, induced by $q_V$, is proper and surjective.
\end{enumerate}
\end{prop}

%\begin{remark}
%The consequence of Proposition \ref{th-partial-prepre} has nothing to with the data of modulus.
%\end{remark}

In the following, we prove Proposition \ref{th-partial-prepre} in several steps.

\subsubsection{A compactification}\label{s5.1.1} %Take any $S'(11) \in \Comp_1 (S(11))$.
Applying Nagata's theorem to the morphism $\phi:\ol{S}(01) \to \ol{S}'(11)$, we obtain a commutative diagram 

\[\begin{xymatrix}{
\ol{S}(01) \ar[r] \ar[d]^f & T \ar[d]_p \\
\ol{S}(11) \ar[r]^g & \ol{S}'(11),
}\end{xymatrix}\]
where the horizontal morphisms are open immersions, and $p$ is a proper surjective morphism. (Since  $\phi$ is \'etale, in particular, quasi-finite, we could also apply Zariski's main theorem.) 
Set
\[
Z(10) := \ol{S}(11) \setminus \ol{S}(10).
\]
Consider the closure $\ol{Z(10)}$ in $\ol{S}'(11)$, given the reduced scheme structure.
Then, in particular, the induced morphism
\[
p_Z : p^{-1} (\ol{Z(10)}) \to \ol{Z(10)}
\]
is proper.

Moreover, set
\[
Z(00) := \ol{S}(01) \setminus \ol{S}(00).
\]
By the assumption that $\ul{S}$ is an elementary Nisnevich square, we have the following canonical identifications of schemes:
\[
Z(00) \iso f^{-1} (Z(10)) := Z(10) \times_{\ol{S}(11)} \ol{S}(01) \iso Z(10).
\]

\

\subsubsection{Shrinking $T$: elimination of the fibers in $p^{-1} (Z(10))$ outside $Z(00)$}

%{\color{red}\large This step is unnecessary in the Zariski case, because we can take $p=\mathrm{id}$, and we have $p^{-1} (Z(10))=Z(00)$. But it seems essential in the Nisnevich case to eliminate fibers of $Z(10)$ outisde $Z(00)$.}

\begin{lemma}\label{l-decomp}
The natural inclusion $Z(00) \to p^{-1} (Z(10))$ is an open and closed immersion.
In particular, there exists an open closed subscheme $W \subset p^{-1} (Z(10))$ such that
\begin{equation}\label{eq-decomp-1}
p^{-1} (Z(10)) = Z(00) \sqcup W.
\end{equation}
\end{lemma}
\begin{proof}
%(This lemma is trivial in the curve case.)
Consider the commutative diagram
\[\begin{xymatrix}{
Z(00) \ar[r] \ar[d]_\wr & \ar[dl]^{\mathrm{separated}} p^{-1} (Z(10)) \\
Z(10).
}\end{xymatrix}\]
The horizontal map is open by the equation
\[
p^{-1} (Z(10)) \cap \ol{S}(01) = Z(00)
\]
and proper, since the morphism $p^{-1} (Z(10)) \to Z(10)$ is separated and an isomorphism is proper.
This finishes the proof.
\end{proof}

Since $p^{-1} (Z(10))$ is a closed subset of $p^{-1} (\ol{S}(11))$, we obtain two disjoint closed subsets $Z(00)$ and $W$ of $p^{-1} (\ol{S}(11))$.
Let $\ol{Z(00)}$ and $\ol{W}$ be their closures in $T$ (with the reduced scheme structures).
Noting that $\ol{Z(00)} \cap \ol{W} \cap  p^{-1} (\ol{S}(11)) = \emptyset$ by (\ref{eq-decomp-1}), the blow-up
\[
\pi :T' := \mathrm{Bl}_{\ol{Z(00)} \cap \ol{W}} T \to T 
\]
is an isomorphism over $p^{-1} (\ol{S}(11))$. Denote by $\widetilde{\ol{Z(00)}}$ and $\widetilde{\ol{W}}$ the strict transforms of $\ol{Z(00)}$ and $\ol{W}$ along $\pi$, respectively.
Then, we have
\[
\widetilde{\ol{Z(00)}} \cap \widetilde{\ol{W}} = \emptyset.
\]
Moreover, regard $Z(00)$ and $W$ as closed subschemes of $\pi^{-1}(p^{-1} (\ol{S}(11)))$, through the isomorphism \[\pi^{-1}(p^{-1} (\ol{S}(11))) \to p^{-1} (\ol{S}(11))\] induced by $\pi$.
Then, the closed subschemes $\widetilde{\ol{Z(00)}}$ and $\widetilde{\ol{W}}$ of $T'$ are equal to the closures of $Z(00)$ and $W$ in $T'$.
%The argument above gives us a new commutative diagram
%\[\begin{xymatrix}{
%\ol{S}(01) \ar[r] \ar[d]^f & T'(01) \ar[d]_{p'} \\
%\ol{S}(11) \ar[r] & \ol{S}'(11),
%}\end{xymatrix}\]
%where the horizontal morphisms are open immersions, and $p':=p\pi$ is a proper surjective morphism such that 
%\[
%(p^{\prime})^{-1} (\ol{S}(11)) \cong p^{-1}(\ol{S}(11)) = Z(00) \sqcup W
%\]
%and the closures of $Z(00)$ and $W$ do not intersect in $T'(01)$.
Since $\pi$ is an isomorphism over $p^{-1}(\ol{S}(11))$, it is in particular an isomorphism over $\ol{S}(01)$ because $\ol{S}(01) \subset p^{-1}(\ol{S}(11))$.
Therefore, the open immersion $\ol{S}(01) \to T$ lifts to an open immersion $\ol{S}(01) \to T'$.
Therefore, we may replace $T$ by $T'$.
After this replacement, we obtain
\begin{equation}\label{empty-001}
\ol{Z(00)} \cap \ol{W} = \emptyset. 
\end{equation}

Define an open subset $U \subset T$ by
\[
U := T \setminus \ol{W},
\]
and let 
\[
q_U :  U \to T \to \ol{S}'(11)
\]
be the natural map.

\begin{prop}\label{nice-over-z}\
\begin{enumerate}
\item We have $q_U^{-1}(Z(10)) = Z(00)$.
\item The closure of $Z(00)$ in $T$ is contained in $U$; in particular, it is proper over $\ol{S}'(11)$.
\end{enumerate}
%(3) Define a closed subset $Z(01) \subset q_U^{-1} (\ol{S}(11))$ by
%\[Z(01) := q_U^{-1} (\ol{S}(11)) \setminus \ol{S}(01) .\]
%and let $\ol{Z(01)}$ be the closure in $U (01)$ with the reduced scheme structure.
%Then, we have
%\[Z(00) \cap Z(01) = \emptyset.
%q_U^{-1} (Z(10)) \cap Z(01) = \emptyset.\]
\end{prop}

\begin{proof}
The map $q_U$ restricted over $\ol{S}(11)$ is equal to 
\[
q_U |_{\ol{S}(11)} : p^{-1} (\ol{S}(11)) \setminus W \hookrightarrow p^{-1} (\ol{S}(11)) \to \ol{S}(11).
\]
Therefore, since $p^{-1} (\ol{S}(Z(10)) \setminus W = Z(00)$ by definition of $W$, we have 
\[q_U^{-1}(Z(10)) = (q_U |_{\ol{S}(11)} )^{-1} (Z(10)) = Z(00).\]
This proves (1), and (2) follows from  \eqref{empty-001} and the properness of $p$.
%
%The equality implies that the closure $\ol{Z(00)}$ of $Z(00)$ in $T$ is contained in $U$.
%Therefore, the closure of $Z(00)$ in $U$ is equal to $\ol{Z(00)}$, hence it is proper over $\ol{S}'(01)$ because $p:T \to \ol{S}'(01)$ is proper.
%This proves the second assertion.
%This finishes the proof.
\end{proof}

\subsubsection{``Separation'' of two closed subsets in $U$}

%{\color{red} In the following arguments, we only use $U$ and $q_U$, which satisfy the assertions in Proposition \ref{nice-over-z}. In other words, we forget $T$.}

%{\color{red} \large When the map $p$ is the identity map, then  the construction below coincides with the one for the Zariski case.}

Define a closed subset $Z(01) \subset q_U^{-1} (\ol{S}(11))$ by
\[
Z(01) := q_U^{-1} (\ol{S}(11)) \setminus \ol{S}(01).
\]
Denote by $\ol{Z(01)}$ and $\ol{Z(00)}$ the closures of $Z(01)$ and $Z(00)$ in $U$ with the reduced scheme structures, respectively.
Then, we have the following

\begin{lemma}\label{proper-int}
\[
\ol{Z(00)} \cap \ol{Z(01)} \cap q_U^{-1} (\ol{S}(11)) = \emptyset.
\]
\end{lemma}

\begin{proof}
Note that $Z(01)$ is closed in $q_U^{-1} (\ol{S}(11))$ by definition.
Moreover, $Z(00)$ is also closed in $q_U^{-1} (\ol{S}(11))$, since by Proposition \ref{nice-over-z}, we have
\[
q_U^{-1} (Z(10)) = Z(00).
\]
Therefore, it suffices to show that $Z(00) \cap Z(01) = \emptyset$.
But this is obvious by $Z(00) \subset \ol{S}(01)$, and by $Z(01) = q_U^{-1} (\ol{S}(11)) \setminus \ol{S}(01)$.
This finishes the proof.
\end{proof}

Consider now the blow-up
\[
\pi_U : U'  := \mathrm{Bl}_{\ol{Z(00)} \cap \ol{Z(01)}} U \to U.
\]
%which is an isomorphism over $q_U^{-1} (\ol{S}(11))$ by Lemma \ref{proper-int}.
 
The strict transforms $\widetilde{\ol{Z(01)}}$ and $\widetilde{\ol{Z(00)}}$  along $\pi_U$ are equal to the closures of $Z(01)$ and $Z(00)$ in $U' $, and we have
\begin{equation}\label{eq000001}
\widetilde{\ol{Z(00)}} \cap \widetilde{\ol{Z(01)}} = \emptyset.
\end{equation}
By Lemma \ref{proper-int}, $\pi_U$ is an isomorphism over $q_U^{-1} (\ol{S}(11))$.
Regard $Z(01)$ and $Z(00)$ as closed subschemes of $\pi_U^{-1} q_U^{-1} (\ol{S}(11))$.
Then, by Proposition \ref{nice-over-z} (1), we have
\[
\pi_U^{-1} q_U^{-1} (Z(10)) = Z(00).
\]
Moreover, the strict transform $\widetilde{\ol{Z(00)}}$ is proper over $\ol{Z(00)}$, hence proper over $\ol{S}'(11)$ by Proposition \ref{nice-over-z} (2).
Therefore, Proposition \ref{nice-over-z} remains true after we replace $U$ by $U'$ and $q_U$ by $q_U \circ \pi_U$, respectively.
After this replacement, the equality (\ref{eq000001}) implies that 
\begin{equation}\label{eq-separation}
\ol{Z(00)} \cap \ol{Z(01)} = \emptyset.
\end{equation}

Define an open subset $V_0 \subset U$ by
\[
V_0 := U \setminus \ol{Z(01)},
\]
and let
\[
q_V : V_0 \to  \ol{S}'(11)
\]
be the natural map.
Then, we have the following
\begin{prop}\label{step-1}
The map $q_V$ satisfies the following conditions:
\begin{enumerate}
\item $q_V^{-1} (\ol{S}(11)) = \ol{S}(01)$. In particular, $q_V$ is \'etale over $\ol{S}(11)$.
\item  The closure of $Z(00)$ in $V_0$ is proper over $\ol{S}'(11)$.
%We have
%\[q_V^{-1} (Z(10)) := Z(10) \times_{\ol{S}'(11)} V_0 (01) = Z(00) \iso Z(10).\]
%and the map $q_V$ restricts to a proper surjective morphism
%\[\ol{Z(00)} \to \ol{Z(10)}.\]
\end{enumerate}
\end{prop}

\begin{proof}
The map $q_V$ restricted over $\ol{S}(11)$ is equal to
\[
q_U^{-1} (\ol{S}(11)) \setminus Z(01) \hookrightarrow q_U^{-1} (\ol{S}(11)) \to \ol{S}(11).
\]
Therefore, by definition of $Z(01)$, we have $q_V^{-1} (\ol{S}(11)) = \ol{S}(01)$.
This proves the assertion (1).
The assertion (2) follows from Proposition \ref{nice-over-z} and the equality (\ref{eq-separation}).
This finishes the proof.
\end{proof}

\subsubsection{``Platification'' of $q_V$}

In this final step, we construct a replacement of $q_V$ so that it becomes flat.
%Now, we forget the data $U$ and $q_U$, and concentrate on $V_0$ and $q_V$.}
Proposition \ref{step-1} (1) implies that $q_V$ is \'etale (in particular, flat) over $\ol{S}(11)$.
Therefore, by the theorem of ``platification'' of Raynaud-Gruson \cite[Th. 5.2.2]{RG}, we can find a blow-up 
\[\pi_S : \widetilde{\ol{S}'(11)} \to \ol{S}'(11)\]
satisfying the following properties:
\begin{enumerate}
 \item Let 
\[
\pi_V : V \to V_0
\] 
be the strict transform of $V_0$ along $\pi_S$.
Then, the morphism 
\[
\tilde{q}_V : V \to \widetilde{\ol{S}'(11)},
\]
which is induced by the universal property of blow-up, is flat.
\item %Let $A \subset \ol{S}'(11)$ be the maximal open subset over which $q_V$ is flat.
The blow up $\pi_S$ is an isomorphism over $\ol{S}(11)$.
\end{enumerate}

Since $\pi_S$ is an isomorphism over $\ol{S}(11)$, the strict transform $\pi_V$ is an isomorphism over $q_V^{-1}(\ol{S}(11)) = \ol{S}(01)$, where the equality follows from Proposition \ref{step-1} (1).
Regard $Z(00)$ as a closed subscheme of $\pi_V^{-1}(\ol{S}(01))$ through the isomorphism
\begin{equation*}
\pi_V^{-1}(\ol{S}(01)) \iso \ol{S}(01)
\end{equation*}
induced by $\pi_V$.
Let $\ol{Z(00)}$ be the closure of $Z(00)$ in $V_0$, and let $\widetilde{\ol{Z(00)}}$ be the closure of $Z(00)$ in $V$.
Then, $\widetilde{\ol{Z(00)}}$ is a strict transform of $\ol{Z(00)}$.
Therefore, we conclude that 
\begin{equation}\label{eq00002}
\text{the strict transform of $Z(00)$ in $V$ is proper over $\widetilde{\ol{S}'(11)}$.}
\end{equation}
Moreover, we have
\begin{multline}\label{eq00003}
\tilde{q}_V^{-1} (\ol{S}(11)) =^1 \tilde{q}_V^{-1} \pi_S^{-1} (\ol{S}(11)) =^2 \pi_V^{-1} q_V^{-1} (\ol{S}(11))\\
 =^3 \pi_V^{-1} (\ol{S}(01)) =^4 \ol{S}(01),
\end{multline}
where $=^1$ follows from Property (2) above, $=^2$ follows from the equality $q_V \pi_V = \pi_S \tilde{q}_V$, $=^3$ follows from Proposition \ref{step-1} (1), and $=^4$ follows from the fact that $\pi_V$ is an isomorphism over $q_V^{-1}(\ol{S}(11))$, which contains $\ol{S}(01)$.

The assertions (\ref{eq00002}) and (\ref{eq00003}) imply that Proposition \ref{step-1} remains true after we replace $V_0$, $\ol{S}'(11)$ and $q_V$ by $V$, $\widetilde{\ol{S}'(11)}$ and $\tilde{q}_V$, respectively.
Therefore, after this replacement, we obtain the flatness of $q_V$.

Finally, $V$ is separated, since it is obtained from $T$ (proper) in \S \ref{s5.1.1} by a composition of open immersions and proper morphisms.
This completes the proof of Proposition \ref{th-partial-prepre}.

\subsection{Construction of a strict extension $\ul{S} \to \ul{S}'$}\label{pf-th-partial-pre}

In this subsection, we use Proposition \ref{th-partial-prepre} to prove

\begin{prop}\label{th-partial-pre}
Let $\ul{S}$ be an elementary Nisnevich square, and take any $N_0(11) \in \Comp_1 (S(11))$.
Then, there exists a strict extension $\ul{S} \to \ul{S}'$ such that $\ul{S}'$ is an elementary Nisnevich square, $S'(11) \in \Comp_1 (S(11))$, $\ol{S}'(11)$ is normal and $S'(11)$ dominates $N_0(11)$.
\end{prop}

In the statement of Proposition \ref{th-partial-prepre}, choose $S'(11)$ dominating $N_0(11)$. We then have the following diagram
\begin{equation}\label{eq5.1}
\begin{gathered}
\xymatrix{
\ol{S}(01) \ar[r] \ar[d]^{\et}_f \ar@{}[rd]|\square& V \ar[d]_{q_V}^{\mathrm{flat}} \ar@{}[dr]|\square & \ar[l] q_V^{-1} (\ol{Z(10)}) \ar[d]_{q_Z} \\
\ol{S}(11) \ar[r] & \ol{S}'(11) & \ar[l] \ol{Z(10)},
}
\end{gathered}
\end{equation}
 where the left square is the one given in Proposition \ref{th-partial-prepre}, and $\ol{Z(10)}$ is the closure of $Z(10) := \ol{S}(11) \setminus \ol{S}(10)$ in $\ol{S}'(11)$ with the reduced scheme structure. The scheme $q_V^{-1} (\ol{Z(10)})$ is defined to be the fiber product $\ol{Z(10)} \times_{\ol{S}'(11)} V$.
We need the following lemma (compare \cite[Cor. 2.2]{lazard} and \cite[Th. 2.7]{oda}): 

\begin{lemma}\label{strong-lemma} 
Let $\phi : X \to S$ be a morphism of schemes and $U \subset S$ an open subset.
Assume the following conditions hold:
\begin{thlist}
\item  $\phi$ is separated, quasi-finite, of finite presentation and flat.
\item $\phi^{-1} (U) \to U$ is an isomorphism.
\item The inclusion $U \to S$ is quasi-compact and scheme-theoretically dense.
\end{thlist} 
Then, $\phi$ is an open immersion. 
\end{lemma}

A proof of Lemma \ref{strong-lemma} is given in the Stacks Project\footnote{Lemma 37.11.5, \url{https://stacks.math.columbia.edu/tag/081M}}.
In fact, the assumption (i) can be replaced by the following weaker condition: $\phi$ is separated, locally of finite type and flat.
However, for our purpose, Lemma \ref{strong-lemma} is enough. We reconstruct a proof in \S \ref{pf-strong-lemma} for the reader's convenience.

\begin{prop}\label{lemma-iso}
The map 
\[
q_Z : q_V^{-1} (\ol{Z(10)}) \to \ol{Z(10)}
\]
 is an isomorphism.
\end{prop}
\begin{proof}

Since $q_V$ is flat by Proposition \ref{th-partial-prepre} (4), so is $q_Z$.
Since $\ol{S}(11) \times_{\ol{S}'(11)} V \cong \ol{S}(01)$ by \eqref{eq5.1}, we have 
\[
Z(10) \times_{\ol{S}'(11)} V \cong Z(10) \times_{\ol{S}(11)} \ol{S}(01) \cong^\dag  Z(10),
\]
where the isomorphism $\cong^\dag$ is obtained by the assumption that $\ul{S}$ is an elementary Nisnevich square.
Therefore, the morphism $q_Z$ is an isomorphism over the dense open subset $Z(10) \subset \ol{Z(10)}$.
Since $\ol{Z(10)}$ is given the reduced scheme structure, the open subset $Z(10) \subset \ol{Z(10)}$ is scheme-theoretically dense. 
Moreover, $q_Z$ is quasi-finite (since it is flat and birational), and separated since $V$ is (see \cite[Ch. I, \S 5.3]{EGA1}).
 Finally, $q_Z$ is of finite presentation since it is a morphism between schemes of finite type (hence of finite presentation) over a field.
Therefore, Lemma \ref{strong-lemma} implies that $q_Z$ is an open immersion.

%It remains to prove that $q_Z$ is surjective.
By Proposition \ref{th-partial-prepre} (3), the map $q_Z$ induces a proper morphism $\ol{Z(00)} \to \ol{Z(10)}$.
Since $Z(00) \to Z(10)$ is an isomorphism, the morphism $\ol{Z(00)} \to \ol{Z(10)}$ is dominant and proper, hence surjective.
This implies that $q_Z$ is surjective, because $\ol{Z(00)} \subset  q_V^{-1} (\ol{Z(10)})$.
This concludes the proof.
\end{proof}

\begin{lemma}\label{lemma-etale}
The flat morphism $q_V$ is \'etale over an open neighborhood of $\ol{Z(10)}$.
In particular, the \'etale locus of $q_V$ contains $q_V^{-1} (\ol{Z(10)})$.
\end{lemma}
\begin{proof}
By Proposition \ref{lemma-iso}, $q_V$ is unramified over $\ol{Z(10)}$.
Since the unramified locus is open, $q_V$ is unramified over an open neighborhood of $\ol{Z(10)}$.
Since $q_V$ is flat, this finishes the proof.
\end{proof}

\begin{definition}
Define an open subset $\ol{S}'(01) \subset V$ by
\[
\ol{S}'(01) := \text{(the \'etale locus of $q_V$)} \subset V.
\]
Set
\[
\ol{S}'(10) := \ol{S}'(11) \setminus \ol{Z(10)}
\]
and
\[
\ol{S}'(00) := \ol{S}'(10) \times_{\ol{S}'(11)} \ol{S}'(01) .
\]
Define for each $(i,j) \in \{(1,0), (0,1), (0,0)\}$
\[
S'(ij) := (\ol{S}'(ij) , S(11)^\infty |_{\ol{S}'(ij)}) .
\]
\end{definition}

\begin{lemma}
The square
\[\begin{xymatrix}{
S'(00) \ar[r] \ar[d] \ar@{}[rd]|\square& S'(01) \ar[d]  \\
S'(10) \ar[r] & S'(11) 
}\end{xymatrix}\] 
is an elementary Nisnevich square extending $\ul{S}$.
Moreover, for all $i,j \in \{0,1\}$, we have 
\begin{equation}\label{l-compatibility-partial}
\ol{S}(ij) = \ol{S}'(ij) \times_{\ol{S}'(11)} \ol{S}(11).
\end{equation}
\end{lemma}

\begin{proof}
The map $\ol{S}'(01) \to \ol{S}'(11)$ is \'etale by definition. 
By Lemma \ref{lemma-iso} and Lemma \ref{lemma-etale}, it is an isomorphism over $\ol{Z(10)} = (\ol{S}'(11) \setminus \ol{S}'(10))_\red$.
This shows that the square is an elementary Nisnevich square. 
We prove the equality $\ol{S}(ij) = \ol{S}'(ij) \times_{\ol{S}'(11)} \ol{S}(11)$.
The case $(i,j)=(0,0)$ is trivial.
The case $(i,j)=(1,0)$ follows from 
\[
\ol{S}'(10) \cap \ol{S}(11) = \ol{S}(11) \setminus \ol{Z(10)} = \ol{S}(11) \setminus Z(10) = \ol{S}(10).
\]

The case $(i,j)=(0,1)$ follows from $\ol{S}(11) \times_{\ol{S}'(11)} V \cong \ol{S}(01)$ (see \eqref{eq5.1}).  
Finally, the case $(i,j)=(0,0)$ follows from
\begin{align*}
\ol{S}'(00) \times_{\ol{S}'(11)} \ol{S}(11) &= (\ol{S}'(10) \times_{\ol{S}'(11)} \ol{S}(11)) \times_{\ol{S}(11)} (\ol{S}'(01) \times_{\ol{S}'(11)} \ol{S}(11)) \\
&= \ol{S}(10) \times_{\ol{S}(11)} \ol{S}(01) = \ol{S}(00).
\end{align*}

This finishes the proof.
\end{proof}

Thus, we have constructed $\ul{S}'$ satisfying the properties in Proposition \ref{th-partial-pre}, except for the normality of $\ol{S}'(11)$.
To obtain the latter, consider modulus pairs
\[
S^{\prime N} (ij) := (\ol{S}'(ij)^N , \text{ the pullback of $S'(ij)^\infty $}) \ i,j=0,1.
\]
By Lemma \ref{lB.1}, we have $\ol{S}'(ij)^N \cong \ol{S}'(ij) \times_{\ol{S}'(11)} \ol{S}'(11)^N$.
Therefore, $S^{\prime N} (ij)$'s form an elementary Nisnevich square $\ul{S}^{\prime N}$.
Since $\ol{S}(ij)$ are normal, the strict extension $\ul{S} \to \ul{S}'$ extends to a strict extension $\ul{S} \to \ul{S}^{\prime N}$.
Thus, $\ul{S}^{\prime N}$ satisfies the properties in Proposition \ref{th-partial-pre}.
This finishes the proof.

\subsection{Construction of a minimal morphism $\ul{S}' \to \ul{N}_0$}\label{modif-partial}

Take $N_0$ as in the statement of Theorem \ref{th-partial}, and $S'$ as given by Proposition \ref{th-partial-pre} with respect to $N_0(11)$.
We will modify $\ul{S}'$ so that it admits a minimal morphism $\ul{S}' \to \ul{N}_0$.
By Proposition \ref{th-partial-pre}, we already have an admissible morphism $S'(11) \to N_0(11)$.
However, it is not necessarily the case for the other corners of the square.
To solve this problem, we consider the following construction.

At least, for each $i,j$, we have a birational map
\[
\ol{S}'(ij) \dashrightarrow \ol{N}_0 (ij),
\]
which is defined on $\ol{S}(ij)$.
Denote by $\Gamma_{ij}$ its graph, i.e., the closure of the graph of $ \ol{S}(ij) \to \ol{N}_0(ij)$ in $\ol{S}'(ij) \times \ol{N}_0(ij)$, with the reduced scheme structure.
Then, we obtain a diagram
\[
\ol{S}'(ij) \xleftarrow{} \Gamma_{ij} \xrightarrow{} \ol{N}_0 (ij).
\]
Consider the composite morphism
\[
f_{ij} : \Gamma_{ij} \to \ol{S}'(ij) \to \ol{S}'(11).
\]

\begin{lemma}\label{l-4.11}
We have $f_{ij}^{-1} (\ol{S}(11)) = \ol{S}(ij)$.
In particular, the morphism $f_{ij}$ is \'etale over $\ol{S}(11)$.\end{lemma}
\begin{proof}
This follows from the fact that the morphism $\Gamma_{ij} \to \ol{S}'(ij)$ is an isomorphism over $\ol{S}(ij)$, and that $\ul{S} \to \ul{S}'$ is a strict extension.
\end{proof}

Lemma \ref{l-4.11} shows that the coproduct
\[
f := \bigsqcup_{(i,j) \neq (1,1)} f_{ij} : \bigsqcup_{(i,j) \neq (1,1)} \Gamma_{ij} \to \ol{S}'(11)
\]
is \'etale (in particular, flat) over $\ol{S}(11)$.
Then the platification theorem \cite[Th. 5.2.2]{RG} applied to $f$ shows that there exists a closed subscheme $Z \hookrightarrow \ol{S}'(11) \setminus \ol{S}(11)$ such that the morphism induced between the blow-ups
\[
\bigsqcup_{(i,j) \neq (1,1)} \tilde{f}_{ij} : \bigsqcup_{(i,j) \neq (1,1)} \mathrm{Bl}_{Z \times_{\ol{S}'(11)} \Gamma_{ij}} (\Gamma_{ij})  \to  \mathrm{Bl}_{Z} (\ol{S}'(11))
\]
is flat.
Consider the following commutative diagram, induced by the universal property of blowing up:
\[\begin{xymatrix}{
\mathrm{Bl}_{Z \times_{\ol{S}'(11)} \Gamma_{ij}} (\Gamma_{ij}) \ar[rr]^{h_{ij}} \ar[rd]_{\tilde{f}_{ij}} &  & \mathrm{Bl}_{Z \times_{\ol{S}'(11)} \ol{S}'(ij)} (\ol{S}'(ij))  \ar[ld]^{g_{ij}} \\
 &\mathrm{Bl}_{Z} (\ol{S}'(11))&
}\end{xymatrix}\]
where $g_{ij}$ is \'etale since we have 
\begin{equation}\label{pullback-blowup}
\mathrm{Bl}_{Z \times_{\ol{S}'(11)} \ol{S}'(ij)} (\ol{S}'(ij)) \cong \ol{S}'(ij) \times_{\ol{S}'(11)} \mathrm{Bl}_{Z} (\ol{S}'(11))\end{equation}
by the \'etaleness of the map $\ol{S}'(ij) \to \ol{S}'(11)$ and by the (trivial) compatibility of blow-up and flat base change.
%\url{https://stacks.math.columbia.edu/tag/085S}.
Since $f_{ij}$ is flat and $g_{ij}$ is \'etale, the horizontal morphism $h_{ij}$ is flat (see for example \cite[Prop. 18.4.9]{EGA4}). %{\color{red} By the same reason as above, I cannot find a reference now. A proof is available here: \url{https://stacks.math.columbia.edu/tag/05B9}.}

However, $h_{ij}$ is a proper birational morphism induced by $\Gamma_{ij} \to \ol{S}'(ij)$, hence it is an isomorphism over the open dense subset $\ol{S}(ij) \subset \mathrm{Bl}_{Z \times_{\ol{S}'(11)} \ol{S}'(ij)} (\ol{S}'(ij))$.
Therefore, Lemma \ref{strong-lemma} implies that $h_{ij}$ is an isomorphism.

Define modulus pairs $S'_1(ij) = (\ol{S}'_1(ij), S'_1(ij)^\infty )$ by
\begin{align*}
\ol{S}'_1(ij) &:= \ol{S}'(ij) \times_{\ol{S}'(11)} \mathrm{Bl}_{Z} \ol{S}'(11) \cong^\dag \mathrm{Bl}_{Z \times_{\ol{S}'(11)} \ol{S}'(ij)} (\ol{S}'(ij)),\\
S'_1(ij)^\infty &:= S'(ij)^\infty \times_{\ol{S}'(ij)} \ol{S}'_1(ij) = S'(ij)^\infty \times_{\ol{S}'(11)} \ol{S}'_1(11) ,
\end{align*}
where $\cong^\dag$ follows from (\ref{pullback-blowup}).
Then, the modulus pairs $S'_1(ij)$ form an elementary Nisnevich square $\ul{S}'_1$.
Since the blow-up $\ol{S}'_1(11) \to \ol{S}'(11)$ is an isomorphism over $\ol{S}(11)$ by the construction, there is a natural strict extension $\ul{S} \to \ul{S}'_1$.
Moreover, we have natural maps on the total spaces
\[
\ol{S}'_1(ij) \cong \mathrm{Bl}_{Z \times_{\ol{S}'(11)} \ol{S}'(ij)} (\ol{S}'(ij)) \xrightarrow{h_{ij}^{-1}} \mathrm{Bl}_{Z \times_{\ol{S}'(11)} \Gamma_{ij}} (\Gamma_{ij}) \to \Gamma_{ij} \to \ol{N}_0(ij).
\]
Moreover, by taking normalization everywhere, we can realize the condition that $\ol{S}'(11)$ is normal.
Here, note that the compatibility between normalization and \'etale base change (see Lemma \ref{lB.1}) ensures that the normalization preserves the property to be an elementary Nisnevich square.

Therefore, by replacing $\ol{S}'$ by $\ol{S}'_1$, we may assume that 
\[
\text{for each $i,j$, the total space $\ol{S}'(ij)$ maps to $\ol{N}_0(ij).$}
\]

However, it is not always the case that these maps induce minimal (or admissible) morphisms ${S}'(ij) \to N_0(ij)$.
Our task in the next subsection is to adjust the moduluses of $\ul{S}'$ and $\ul{N}_0$ to ensure that these maps on the total spaces induce minimal morphisms.

\subsection{Modification of moduluses of $\ul{S}'$ and $\ul{N}_0$}

The argument will be divided into several steps.
First, note that we may assume the following condition without loss of generality:
\[
\ol{S}'(11) = \ol{N}_0(11).
\]
Indeed, consider the modulus pairs 
\[
N'_0(ij) := (\ol{N}_0(ij) \times_{\ol{N}_0(11)} \ol{S}'(11), \text{ the pullback of $N_0(ij)^\infty$}), \ (i,j=0,1).
\]
These modulus pairs naturally form a square $\ul{N}'_0 \in \Comp_1(\ul{S})$ dominating $\ul{N}_0$, because the projection $\ol{N}'_0(ij) \to \ol{N}_0(ij)$ is an isomorphism over $\ol{S}(ij)$ for each $i,j$ since $\ol{S}'(11) \to \ol{N}_0(11)$ is an isomorphism over $\ol{S}(11)$ by Lemma \ref{no-extra-fiber}.
Moreover, the morphisms of schemes $\ol{S}'(ij) \to \ol{N}_0(ij)$ canonically lift to $\ol{S}'(ij) \to \ol{N}'_0(ij)$.
Therefore, by replacing $\ul{N}_0$ by $\ul{N}'_0$, we obtain the equality above.

\subsubsection{Enlargement of the modulus of $\ul{S}'$.}

First, we enlarge the modulus of $\ul{S}'$, preserving the minimality of $\ul{S} \to \ul{S}'$, as follows.
%fix a positive integer $m$.

Define $Z:=\ol{S}' (11) \setminus \ol{S}(11)$, and regard $Z$ as a closed subscheme of $\ol{S}' (11) = \ol{N}_0 (11)$ with the reduced scheme structure.
%Note that $Z$ does not intersect with $\ol{S}(11)$.
Consider the blow-up
\begin{equation}\label{eq5.2}
\mathrm{Bl}_Z \ol{N}_0(11) \to \ol{N}_0(11) ,
\end{equation}
which is an isomorphism over $\ol{S}(11)$ since $Z \cap \ol{S}(11) = \emptyset$.
Replacing the total spaces $\ol{N}_0(ij)$ and $\ol{S}'(ij)$ by their strict transforms along this blow-up, and pulling back every modulus, we may assume that $Z$ is an effective Cartier divisor on $\ol{N}_0$.
Here, note that for each $i,j$, the strict transform $\tilde{\ol{S}'}(ij)$ of $\ol{S}'(ij)$ is given by the fiber product $\ol{S}'(ij) \times_{\ol{N}_0(11)} \mathrm{Bl}_Z \ol{N}_0(11)$, which ensures that $\tilde{\ol{S}'}(ij)$'s form an elementary Nisnevich square.

Define effective Cartier divisors on $\ol{S}' (ij)$ by
\[
Z_{ij} := Z \times_{\ol{S}'(11)} \ol{S}' (ij).
\]
Note that 
\begin{equation}\label{eq-sup-1}
|Z_{ij}|= (\ol{S}' (11) \setminus \ol{S}(11)) \times_{\ol{S}'(11)} \ol{S}' (ij) =^\dag \ol{S}' (ij) \setminus \ol{S}(ij),
\end{equation}
where the equality $=^\dag$ follows from (\ref{l-compatibility-partial}).

\begin{prop}\label{enlarging-general}
There exists a positive integer $m$ satisfying 
\[
N_0 (ij)^\infty |_{\ol{S}'(ij)} \leq S' (ij)^\infty + m Z_{ij}
\]
for each $i,j$, where $N_0 (ij)^\infty |_{\ol{S}'(ij)}$ denotes the pullback of the effective Cartier divisor $N_0 (ij)^\infty$ by the dominant morphism $\ol{S}'(ij) \to \ol{N}_0(ij)$.
\end{prop}

\begin{proof}
The restrictions of the Cartier divisor $S' (ij)^\infty - N_0 (ij)^\infty |_{\ol{S}'(ij)}$ to $\ol{S}(ij)$ is zero.
Therefore, by the equality (\ref{eq-sup-1}), Lemma \ref{m-i-l} implies that we can find positive integers $m_{ij}$ such that 
\[
N_0 (ij)^\infty |_{\ol{S}'(ij)} \leq S' (ij)^\infty + m_{ij} Z_{ij}
\]
holds. Take $m:=\max_{i,j} m_{ij}$. This finishes the proof.
\end{proof}

Take $m$ as in Proposition \ref{enlarging-general}.
Define modulus pairs $S'_1 (ij)$ by
\[
S'_1 (ij) := \left(\ol{S}' (ij), S'(ij)^\infty + mZ_{ij} \right).
\]
Since by definition we have
$\ol{S}'(ij) \times_{\ol{S}'(11)} Z_{11} = Z_{ij} ,$
the modulus pairs $S'_1(ij)$'s form an elementary Nisnevich square $\ul{S}'_1$.

\begin{prop}\label{enlarging-general-2}
The map $\ul{S} \to \ul{S}'_1$ is a strict extension, and we have
\begin{equation}\label{eq-4.13}
N_0 (ij)^\infty |_{\ol{S}'_1(ij)} \leq S'_1(ij)^\infty ,
\end{equation}
where $N_0 (ij)^\infty |_{\ol{S}'_1(ij)}$ denotes the pullback of the Cartier divisor $N_0 (ij)^\infty$ by the dominant morphism $\ol{S}'_1(ij) \to \ol{N}_0(ij)$.
Hence the morphism $\ul{S}\to \ul{N}_0$ extends to a morphism $\ul{S}'_1\to \ul{N}_0$.
\end{prop}

\begin{proof}
The minimality of the map follows from $|Z_{ij}| \cap \ol{S}(ij) = \emptyset$ and the minimality of $\ul{S} \to \ul{S}'$.
The equality $\ol{S}'_1(ij) \times_{\ol{S}'_1(11)} \ol{S}(11) = \ol{S}(ij)$ follows from $\ol{S}'_1(ij) = \ol{S}'(ij)$.
Therefore, $\ul{S} \to \ul{S}'_1$ is a strict extension.
The last assertion is immediate by Proposition \ref{enlarging-general}.
\end{proof}

Up to pulling back $N_0$ by \eqref{eq5.2}, we may assume that the condition  $\ol{S}'_1(11) = \ol{N}_0 (11)$ continues to hold.

\subsubsection{Creation of Cartier divisors}

%In the following, we fix a positive integer $m$ as in Proposition \ref{enlarging-general}.
%Then, $\ul{S} \to \ul{S}'_1$ is minimal, $\ul{S}'_1$ is an elementary Nisnevich square, and we have $N_0 (ij)^\infty |_{\ol{S}'(ij)} \leq S'_1 (ij)^\infty $ for each $i,j \in \{0,1\}$.
%Note that $S_1' (11)^\infty$ is an effective Cartier divisor on $\ol{S}'(11) = \ol{N}_0(11)$.

\begin{lemma}\label{l-harmless}
For all $i,j \in \{0,1\}$, define a closed subscheme $B_{ij}$  of $N_0(ij)^\infty$ by
\[
B_{ij} := S'_1(11)^\infty \times_{\ol{N}_0(11)} N_0(ij)^\infty. %= (S'_1(11)^\infty \times_{\ol{N}_0(11)} \ol{N}_0(ij))  \times_{\ol{N}_0(ij)} N_0(ij)^\infty .
\]
Then, for each $i,j\in \{0,1\}$, we have
\begin{equation}\label{5.16eq1}
B_{ij} |_{\ol{S}'_1(ij)}  = N_0(ij)^\infty |_{\ol{S}'_1(ij)} .
\end{equation}
%and
%\begin{equation}\label{5.16eq2}
%B_{ij} \times_{\ol{N}_0(ij)} \ol{S}(ij)  = S(ij)^\infty .
%\end{equation}
In particular, the left hand side of the equality is an effective Cartier divisor on $\ol{S}'_1(ij)$.
\end{lemma}

\begin{proof}
%{\color{red}
%The equality \eqref{5.16eq2} follows from \eqref{5.16eq1} and from the minimality of $S(ij) \to N_0(ij)$.
%We prove \eqref{5.16eq1}. 
%First, note that the definition of $B_{ij}$ implies
%\[B_{ij} \times_{\ol{N}_0(ij)} \ol{S}'_1(ij) = \left( S'_1(11)^\infty \times_{\ol{N}_0(11)} \ol{S}'_1(ij) \right)\times_{\ol{S}'_1(ij)} \left( N_0(ij)^\infty \times_{\ol{N}_0(ij)} \ol{S}'_1(ij) \right).\]
%Recalling the assumption that $\ol{S}'_1(11) = \ol{N}_0(11)$, the first factor of the fiber product on the right hand side is equal to $S'_1(ij)^\infty$. Therefore, we have
%\[B_{ij} \times_{\ol{N}_0(ij)} \ol{S}'_1(ij) = S'_1(ij)^\infty \times_{\ol{S}'_1(ij)} \left( N_0(ij)^\infty \times_{\ol{N}_0(ij)} \ol{S}'_1(ij) \right).\]
%Je n'arrivais pas \`a suivre ta d\'emonstration, donc j'ai \'ecrit ceci (note aussi le changement dans l'\'enonc\'e du lemme).
%{\color{red} \large I do not understand this equality. We have $N_0 (11)^\infty \leq S'_1(11)^\infty$ by Proposition \ref{enlarging-general-2}, but how do you get the opposite direction?}
Since $S'_1$ is an elementary Nisnevich square, we have
\[S'_1(11)^\infty|_{\ol{S}'_1(ij)} = S'_1(ij)^\infty\]
and the claim follows from Proposition \ref{enlarging-general-2}. %implies that the right hand side is equal to $N_0(ij)^\infty \times_{\ol{N}_0(11)} \ol{S}'_1(ij) = N_0(ij)^\infty |_{\ol{S}'_1(ij)}$.
%This finishes the proof.}
\end{proof}

\begin{lemma}\label{l5.1}
There exists a morphism $\ul{\tilde{N}} _0\by{\pi} \ul{N}_0$ in $\Comp_1^\mor(\ul{S})$ such that
\begin{enumerate}
\item The morphism $\ul{S}'_1\to \ul{N}_0$ lifts to $\ul{\tilde{N}} _0$.
\item $\pi^*B_{ij}$ is an effective Cartier divisor for all $i,j$.
\item The total space of $\tilde{N}_0(ij)$ is normal for any $(i,j)$.
\end{enumerate}
\end{lemma}
%Since $\ol{N}_0(11)=\ol{S}'_1(11)$, Lemma \ref{l-harmless} implies that 

\begin{proof}
The equality $\ol{S}'_1(11) = \ol{N}_0 (11)$ implies the inequality $S'_1(11)^\infty\ge N_0 (11)^\infty$.
Thus, $B_{11}=N_0(11)^\infty$ is already an effective Cartier divisor on $\ol{N}_0(11)$. We set $\tilde{N} _0(11)=N_0(11)$.

Next, let $(i,j)$ be one of $(1,0)$ or $(0,1)$, and define modulus pairs $\tilde{N}_0(ij)$ and $\tilde{N}_0(00)$ by
\begin{align*}
\tilde{N}_0(ij) &:= (\mathrm{Bl}_{B_{ij}} \ol{N}_0 (ij),\text{ the pullback of $N_0(ij)^\infty$}), \\
\tilde{N}_0(00) &:= (\ol{N}_0(00) \times_{\ol{N}_0(ij)} \tilde{\ol{N}}_0(ij) ,\text{ the pullback of $N_0(00)^\infty$}).
\end{align*}
Then, by Lemma \ref{l-harmless}, the morphism $S'_1(ij) \to N_0(ij)$ lifts to $S'_1(ij) \to \tilde{N}_0(ij)$.
Therefore, the morphism $S'_1(00) \to N_0(00)$ lifts to $S'_1(00) \to \tilde{N}_0(00)$.
Moreover, we have
\[
B_{ij} \times_{\ol{N}_0 (ij)} \tilde{\ol{N}}_0(ij) = S'_1(11)^\infty \times_{\ol{N}_0(11)} \tilde{N}_0(ij)^\infty .
\]
Denote by $\tilde{\ul{N}}_0$ the square obtained from $\ul{N}_0$ by replacing $N_0(ij)$ and $N_0(00)$ by $\tilde{N}_0(ij)$ and $\tilde{N}_0(00)$, respectively.
Then, $\tilde{\ul{N}}_0$ dominates $\ul{N}_0$, and the map $\ul{S}'_1 \to \ul{N}_0$ lifts to a morphism $\ul{S}'_1 \to \tilde{\ul{N}}_0$.
Therefore, by replacing $\ul{N}_0$ by $\tilde{\ul{N}}_0$, we may assume that $B_{ij}$ is an effective Cartier divisor on $\ol{N}_0(ij)$.

After this replacement, we apply the same procedure to $(i',j') \in \{(1,0),(0,1)\} - \{(i,j)\}$.
Then, we may assume that $B_{ij}$ is an effective Cartier divisor on $\ol{N}_0(ij)$ for $(i,j)=(1,0),(0,1)$.

Now, we reset the notation, and treat the case $(i,j)=(0,0)$.
Define a modulus pair $\tilde{N}_0(00)$ by
\[
\tilde{N}_0(00) := (\mathrm{Bl}_{B_{00}} \ol{N}_0 (00),\text{ the pullback of $N_0(00)^\infty$}).
\]
Then, by Lemma \ref{l-harmless}, the morphism $S'_1(00) \to N_0(00)$ lifts to $S'_1(00) \to \tilde{N}_0(00)$.
Moreover, we have 
\[
B_{00} \times_{\ol{N}_0 (00)} \tilde{\ol{N}}_0(00) = S'_1(11)^\infty \times_{\ol{N}_0(11)} \tilde{N}_0(00)^\infty .
\]
Therefore, replacing $N_0(00)$ by $\tilde{N}_0(00)$, we may assume that $B_{00}$ is an effective Cartier divisor on $\ol{N}_0(00)$.

Finally, we may assume that $\ol{N}_0(ij)$ is normal without loss of generality, just by replacing each $N_0(ij)$ by 
\[
N_0(ij)^N := (\ol{N}_0(ij)^N , \text{ the pullback of $N_0(ij)^\infty $}).
\]
Note that $N_0(ij)^N$'s form a square $\ul{N}_0^N$ which dominates $\ul{N}_0$, and that the morphism $\ul{S}'_1 \to \ul{N}_0$ lifts to a morphism $\ul{S}'_1 \to \ul{N}_0^N$ by the normality of $\ol{S}'_1(ij)$'s.
\end{proof}

\subsubsection{Enlargement of the modulus of $\ul{N}_0$.}
In the following, we replace $\ul{N}_0$ by $ \ul{\tilde N}_0$ as in Lemma \ref{l5.1}, hence assume that the subscheme $B_{ij}$ of Lemma \ref{l-harmless} is an effective Cartier divisor on $\ol{N}_0(ij)$ and that the latter is normal, for each $i,j\in\{0,1\}$.

\ 

Define modulus pairs $N_1(ij) = (\ol{N}_1(ij),N_1(ij)^\infty )$ by
\begin{align*}
\ol{N}_1(ij) :&= \ol{N}_0(ij), \\
N_1 (ij)^\infty :&= \sup  (S'_1(11)^\infty \times_{\ol{N}_0(11)} \ol{N}_0(ij) ,N_0 (ij)^\infty ),
\end{align*}
where the sup is taken as Weil divisors, but it is also the sup as Cartier divisors by Lemma \ref{sup-lemma-pre}, the normality of $\ol{N}_0(ij)$ and the assumption that $B_{ij}$ is an effective Cartier divisor on $\ol{N}_0(ij)$.

\begin{lemma}\label{sup-minimality-general}
For each $i,j \in \{0,1\}$, we have
\[N_1 (ij)^\infty |_{\ol{S}'(ij)} = S'_1 (ij)^\infty ,\]
where $N_1 (ij)^\infty |_{\ol{S}'_1(ij)}$ denotes the pullback of the Cartier divisor $N_1 (ij)^\infty$ by the dominant morphism $\ol{S}'_1(ij) \to \ol{N}_1(ij) = \ol{N}_0(ij)$.
\end{lemma}

\begin{proof}
By Lemma \ref{sup-lemma-pre} (3), we have
\begin{align*}
N_1 (ij)^\infty |_{\ol{S}'(ij)} &= \sup  (S_1' (11)^\infty \times_{\ol{N}_0(11)} \ol{N}_0(ij) |_{\ol{S}'(ij)} ,N_0 (ij)^\infty |_{\ol{S}'(ij)} ) \\
&=^1 \sup  (S'_1(ij)^\infty ,N_0 (ij)^\infty |_{\ol{S}'(ij)} ) \\
&=^2 S'_1(ij)^\infty
\end{align*}
where the equality $=^1$ follows from
\[
S_1' (11)^\infty \times_{\ol{N}_0(11)} \ol{N}_0(ij) |_{\ol{S}'(ij)} = S_1' (11)^\infty \times_{\ol{N}_0(11)} \ol{S}'_1(ij) ,
\]
and from the minimality of the morphism $S'_1(ij) \to S'_1(11)$.
The equality $=^2$ follows from Proposition \ref{enlarging-general-2}.
This finishes the proof.
\end{proof}

We summarize the results of this subsection in the following proposition.

\begin{prop}
The modulus pairs $N_1(ij)$'s form a square $\ul{N}_1$ in $\MCor^\Sq$ such that $\ul{N}_1 \in \Comp_1^\mor (\ul{S})$ and there exists a morphism $\ul{N}_1 \to \ul{N}_0$ in $\Comp_1 (\ul{S})$.
Moreover, the maps on the total spaces $\ol{S}'_1 (ij) \to \ol{N}_0(ij) = \ol{N}_1(ij)$ induce a minimal morphism $\ul{S}'_1 \to \ul{N}_1$.
\end{prop}

\begin{proof}
Let $(i,j) \to (i',j')$ be a morphism in $\mathbf{Sq}$, and let $f : \ol{N}_0 (ij) \to \ol{N}_0 (i'j')$ be the corresponding morphism on total spaces.
Then, we have 
%$S'_1 (ij)^\infty \geq S'_1 (i'j')^\infty$ and 
$N_0 (ij)^\infty \geq f^\ast N_0 (i'j')^\infty$ since $\ul{N}_0 \in \mathbf{MCor}^\mathbf{Sq}$.
This implies that 
\begin{align*}
N_1 (ij)^\infty &= \sup \{ S_1' (11)^\infty \times_{\ol{N}_0(11)} \ol{N}_0(ij) , N_0 (ij)^\infty \} \\
&\geq \sup \{ f^\ast (S_1' (11)^\infty \times_{\ol{N}_0(11)} \ol{N}_0(i'j')) , f^\ast N_0 (i'j')^\infty \} = f^\ast N_1 (i'j')^\infty ,
\end{align*}
where the last equality follows from Lemma \ref{sup-lemma-pre} (3).
This proves that the square $\ul{N}_1$ is well-defined.
The existence of the map $\ul{N}_1 \to \ul{N}_0$ is obvious by the construction of $\ul{N}_1$.
Lemma \ref{sup-minimality-general} implies that the morphism $\ul{S}'_1 \to \ul{N}_1$ is minimal.
in particular,  $\ul{S} \to \ul{N}_1$ is also minimal, which implies $ \ul{N}_1 \in \Comp_1 (\ul{S})$.
This finishes the proof.
\end{proof}

Thus, we have finished the proof of Theorem \ref{th-partial}.

\section{Existence of MV-compactifications}\label{section-MV}

The aim of this section is to prove Theorem \ref{existence-MV}. Thus, throughout, $\ul{S}$ is an elementary Nisnevich square with $\ol{S}(11)$ proper.

\subsection{A remark and a lemma}

\begin{remark}\label{mv-sequence}
We recall a basic discussion from \cite[Lemma 4.3.2, Lemma 4.3.3]{motmod}.
Consider any compactification $\ul{N}\in \Comp_1(S)$. Then, the associated sequence
\begin{equation*}
0 \to \Z_\tr N(00) \to \Z_\tr N(10) \oplus \Z_\tr N(01) \to \Z_\tr N(11) \to 0
\end{equation*}
is automatically exact at $\Z_\tr N(00)$ and $\Z_\tr N(11)$ in $\mathbf{MNST}$.
Indeed, the injectivity of $\Z_\tr N(00) \to \Z_\tr N(10)$ is trivial, and the surjectivity of $\Z_\tr N(10) \oplus \Z_\tr N(01) \to \Z_\tr N(11)$ follows from the surjectivity of $\Z_\tr S(10) \oplus \Z_\tr S(01) \to \Z_\tr S(11) = \Z_\tr N(11)$, where the equality $S(11) = N(11)$ is a consequence of the properness of $\ol{S}(11)$.
\end{remark}

Let $N_0 \in \Comp_1 (\ul{S})$ be a compactification.
We must construct an MV-compactification $\ul{N}_1 \in \Comp_1 (\ul{S})$ which admits a morphism $\ul{N}_1 \to \ul{N}_0$ in $\MCor^{\mathbf{Sq}}$. By Lemma \ref{l-cofinality-minimal}, we may assume that $\ul{N}_0 \in \Comp_1^{\min} (\ul{S})$,  where we take the notation of \eqref{name} (see Definition \ref{defn-minimal-sq}).  Before beginning the proof, we prepare an elementary lemma.

\begin{lemma}\label{elementary-composition}
Let $X, Y$ and $Z$ be smooth schemes over $k$, $\alpha \in \Cor (X,Y)$ an elementary finite correspondence, and $f : Y \to Z$ be a morphism of $k$-schemes.
Denote by $\Gamma_f \in \Cor (Y,Z)$ the graph of $f$, which is regarded as an elementary finite correspondence. 
Set $\beta := (\mathrm{id}_X \times f)(\alpha ) \subset X \times Z$.
Then, $\beta$ coincides with the support of the divisor $\Gamma_f \circ \alpha \in \Cor (X,Z)$, where $\circ$ denotes the composition in the category of finite correspondences $\Cor$.
\end{lemma}

\begin{proof}
The composition $\Gamma_f \circ \alpha$ is defined as the pushforward of the cycle $(\alpha \times Z) \cdot (X \times \Gamma_f) \subset X \times Y \times Z$ by the projection $X \times Y \times Z \to X \times Z$, where $\cdot $ denotes the intersection product (see for example \cite[\S 1]{mvw}).
Therefore, the support $|\Gamma_f \circ \alpha|$ is equal to the image of the set-theoretic map $(\alpha \times Z) \cap (X \times \Gamma_f) \to X \times Y \times Z \to X \times Z$, which is nothing but $\beta$. This finishes the proof.
\end{proof}

\subsection{Construction of a compactification}\label{subsec-6.2}

In this subsection, we construct an object $\ul{N}_1\in \Comp_1(\ul{S})$ dominating $\ul{N}_0$; the main result will be that it is an MV-compactification as in the statement of Theorem \ref{existence-MV}.

\subsubsection{A splitting} For the reader's convenience, we reproduce the square \eqref{name} here:
\begin{equation}\label{name0}
\begin{CD}
\ol{N}_0(00)@>h_u>> \ol{N}_0(01)\\
@Vv_lVV @Vv_r VV\\
\ol{N}_0(10)@>h_d>> \ol{N}_0(11).
\end{CD}
\end{equation}

As in Section \ref{section-partial}, define closed subschemes 
\[
Z(10) := \ol{S}(11) \setminus \ol{S}(10) \subset \ol{S}(11)
\]
and 
\[
Z(00) := \ol{S}(01) \setminus \ol{S}(00) \subset \ol{S}(01),
\]
with the reduced scheme structures.
Since $\ol{S}$ is an elementary Nisnevich square, we have the following identifications of schemes:
\begin{equation}\label{eq-identification}
Z(00) = Z(10) \times_{\ol{S}(11)} \ol{S}(01) \iso Z(10).
\end{equation}

%Take any $\ul{N}_0 \in \Comp_1^{\min} (\ul{S})$.
%, i.e.,  the morphisms in the square are given by morphisms of schemes on total spaces, and writing
%\begin{equation*}
%\begin{xymatrix}{
%\ol{N}_0(00) \ar[r]^{h_u} \ar[d]_{v_l} \ar[rd]^p &  \ol{N}_0(01) \ar[d]^{v_r}\\
%\ol{N}_0(10) \ar[r]^{h_d} & \ol{N}_0(11),
%}\end{xymatrix}\end{equation*}
%we have $N_0(00)^\infty = \sup \{h_u^\ast N_0(01)^\infty , v_l^\ast N_0(10)^\infty \}$.
%Here, \[p := v_r \circ h_u = h_d \circ v_l.\]
Since $\ol{S}(11)$ is proper, we have $S(11) = N_0 (11)$.

By the same proof as that of Lemma \ref{l-decomp}, we can find a closed open subset $W \subset v_r^{-1} (Z(10))$ such that 
\begin{equation}\label{eq0005}
v_r^{-1} (Z(10)) = Z(00) \sqcup W.
\end{equation}

\begin{lemma}\label{l-w}
The set $W$ is a closed subset of $\ol{N}_0(01) \setminus \ol{S}(01)$.
\end{lemma}

\begin{proof}
Since $W$ is closed in $v_r^{-1} (Z(10))$, it is closed in $\ol{N}_0(01)$.
So, it suffices to prove that $W \cap \ol{S}(01) = \emptyset.$.
Since $\ul{S}$ is an elementary Nisnevich square, we have
\[
v_r^{-1}(Z(10)) \cap \ol{S}(01) = Z(10) \times_{\ol{S}(11)} \ol{S}(01) = Z(00).
\]
By \eqref{eq0005}, we obtain
\[
(Z(00) \sqcup W) \cap \ol{S}(01) = Z(00),
\]
which implies $W \cap \ol{S}(01) = \emptyset.$
This finishes the proof.
\end{proof}

%(We do not need this blue argument in the curve case.) 

\subsubsection{Creation of a Cartier divisor}
We regard $W$ as a closed subscheme of $\ol{N}_0(01)$ with the reduced scheme structure. 
We reduce to the case where $W$ is an effective Cartier divisor on $\ol{N}_0(01)$, as follows. By Lemma \ref{l-w}, the blow-up 
\[
\tilde{\ol{N}}_0(01) := \mathrm{Bl}_W (\ol{N}_0(01)) \to \ol{N}_0(01)
\]
is an isomorphism over $\ol{S}(01)$.
Define 
\[
\tilde{\ol{N}}_0(00) := \ol{N}_0(00) \times_{\ol{N}_0 (01)} \tilde{\ol{N}}_0(01).
\]
Then, the projection $\tilde{\ol{N}}_0(00) \to \ol{N}_0(00)$ is an isomorphism over $\ol{S}(00)$.
Therefore, by replacing $\ol{N}_0(01)$ and $\ol{N}_0(00)$ with $\tilde{\ol{N}}_0(01)$ and $\tilde{\ol{N}}_0(00)$, respectively, and by pulling-back the moduluses, we may, and do, assume that $W$ is an effective Cartier divisor on $\ol{N}_0(01)$.

\subsubsection{A majoration} Note that $Z(00)$ is also a closed subscheme of $\ol{N}_0(01)$ by $(\ref{eq0005})$.
Define closed subschemes $\tilde{Z}(00), \tilde{W}$ of $\ol{N}_0(00)$ by
\begin{align}
\tilde{Z}(00) &:= h_u^{-1} (Z(00)) = Z(00) \times_{\ol{N}_0(01)} \ol{N}_0(00), \label{z-tilde}\\
\tilde{W} &:= h_u^{-1} (W) = W \times_{\ol{N}_0(01)} \ol{N}_0(00) . \label{w-tilde}
\end{align}
Then, we have
\begin{equation}\label{eq0006}
p^{-1} (Z(10)) = h_u^{-1} v_r^{-1} (Z(10)) = h_u^{-1} (Z(00) \sqcup W) = \tilde{Z}(00) \sqcup \tilde{W}.
\end{equation}

\begin{lemma}\label{l-m}
There exists a positive integer $m$ such that 
\[
v_l^\ast N_0(10)^\infty |_{\ol{N}_0(00) \setminus \tilde{Z}(00)} \leq h_u^\ast (N_0(01)^\infty + mW) |_{\ol{N}_0(00) \setminus \tilde{Z}(00)} .
\]
\end{lemma}
\begin{proof}
First, we prove
 
\begin{claim}\label{claim76}
$v_l^\ast N_0(10)^\infty |_{p^{-1}(\ol{S}(10))}  \leq h_u^\ast N_0(01)^\infty  |_{p^{-1}(\ol{S}(10))}.$
\end{claim}

\begin{proof}[Proof of Claim]
Since the morphism $N_0(01) \to N_0(11)$ is admissible, we have $N_0(01)^\infty \geq v_r^\ast N_0(11)^\infty$.
Since $p=v_r \circ h_u$, we obtain $h_u^\ast N_0(01)^\infty\allowbreak \geq p^\ast N_0(11)^\infty$.
Therefore, it suffices to prove that 
\[
p^\ast N_0(11)^\infty |_{p^{-1} (\ol{S}(10))} = v_l^\ast N_0(10)^\infty |_{p^{-1} (\ol{S}(10))}.
\]
Since $p=h_d \circ v_l$, we are reduced to showing
\[
h_d^\ast N_0(11)^\infty |_{h_d^{-1} (\ol{S}(10))} = N_0(10)^\infty |_{h_d^{-1} (\ol{S}(10))}.
\]
By applying Lemma \ref{no-extra-fiber} to the morphism $h_d : \ol{N}_0(10) \to \ol{N}_0(11)$ and the dense open subset $\ol{S}(10) \subset \ol{N}_0(10)$, we have $h_d^{-1}(\ol{S}(10)) = \ol{S}(10)$.
Therefore, it suffices to prove 
\[
h_d^\ast N_0(11)^\infty |_{\ol{S}(10)} = N_0(10)^\infty |_{\ol{S}(10)} .
\]
Since $N_0(11) = S(11)$, both sides of the above equality coincides with $S(10)^\infty $ by the minimality of the morphism $S(10) \to S(11)$.
This finishes the proof of Claim \ref{claim76}.
\end{proof}

By (\ref{eq0006}), we have
\begin{multline*}
p^{-1}(\ol{S}(10)) = p^{-1}(\ol{S}(11) \setminus Z(10))\\ = \ol{N}_0(00) \setminus p^{-1}(Z(10)) = \ol{N}_0(00) \setminus (\tilde{Z}(00) \sqcup \tilde{W}).
\end{multline*}

Therefore, Claim \ref{claim76} says that
 \[
 (h_u^\ast N_0(01)^\infty - v_l^\ast N_0(10)^\infty ) |_{\ol{N}_0(00) \setminus (\tilde{Z}(00) \sqcup \tilde{W})} \geq 0.
 \]
Then, Lemma \ref{m-i-l} implies that there exists a positive integer $m$ such that 
\[
(h_u^\ast N_0(01)^\infty + m\tilde{W} - v_l^\ast N_0(10)^\infty ) |_{\ol{N}_0(00) \setminus \tilde{Z}(00)} \geq 0.
\]
This finishes the proof of Lemma \ref{l-m}.
\end{proof}

\subsubsection{Construction of $\ul{N}_1$} In the following, we fix $m \geq 1$ as in Lemma \ref{l-m}.
%(We do not need this discussion in the curve case.)

Define an effective Cartier divisor $N_1(01)^\infty $ on $\ol{N}_0(01)$ by
\begin{equation}\label{defn-n1-mod}
N_1(01)^\infty := N_0(01)^\infty + mW
\end{equation}
(see \eqref{eq0005} for the definition of $W$, and Lemma \ref{l-m} for $m$).
Define a modulus pair $N_1(01)$ by
\begin{equation}\label{defn-n1-pair}
N_1(01) := (\ol{N}_0(01), N_1(01)^\infty).
\end{equation}

\begin{lemma}\label{l-b-1}
We have the following equality of closed subschemes of $\ol{S}(01)$:
\[
v_l^\ast N_0(10)^\infty \cap \ol{S}(00) = S(00)^\infty = h_u^\ast N_1(01)^\infty \cap \ol{S}(00).
\]
In particular, we have
\[
(v_l^\ast N_0(10)^\infty \times_{\ol{N}_1(00)} h_u^\ast N_1(01)^\infty ) \cap \ol{S}(00) = S(00)^\infty .
\]
\end{lemma}

\begin{proof}
The equality $v_l^\ast N_0(10)^\infty \cap \ol{S}(00) = S(00)^\infty$ follows from the minimality of the morphism $S(00) \to S(10) \to N_0(10)$.
By Lemma \ref{l-w}, the morphism $S(01) \to N_1(01)$ remains minimal. Therefore, the equality $h_u^\ast N_1(01)^\infty \cap \ol{S}(00) = S(00)^\infty$ follows from the minimality of the composite $S(00) \to S(01) \to N_1(01)$.
This finishes the proof.
\end{proof}

By Lemma \ref{l-b-1}, the blow-up
\[
\mathrm{Bl}_{v_l^\ast N_0(10)^\infty \times_{\ol{N}_1(00)} h_u^\ast N_1(01)^\infty} \ol{N}_0(00) \to \ol{N}_0(00)
\]
is an isomorphism over $\ol{S}(01)$.
Note that $\ol{S}(00)$ is normal by Remark \ref{r2.1}.
Therefore, the normalized blow-up of $\ol{N}_0(00)$ along the closed subscheme $v_l^\ast N_0(10)^\infty \times_{\ol{N}_1(00)} h_u^\ast N_1(01)^\infty$ is an isomorphism over $\ol{S}(00)$.
So, replacing $\ol{N}_0(00)$ by the blow-up and by pulling back the moduluses, we may assume that 
\begin{multline}\label{asumption6.6}
v_l^\ast N_0(10)^\infty \times_{\ol{N}_1(00)} h_u^\ast N_1(01)^\infty \text{ is an effective Cartier divisor}\\
\text{on } \ol{N}_0(00).
\end{multline}

Define modulus pairs $N_1(11),N_1(10)$ and $N_1(00)$ by
\begin{align*}
N_1(ij) &= N_0(ij) \text{ for $(i,j) \in \{(1,1),(1,0)\}$}, \\
%N_1(01) &= (\ol{N}_0,N_1(01)) , \\
N_1(00) &= (\ol{N}_0(00),\sup \{v_l^\ast N_0(10)^\infty , h_u^\ast N_1(01)^\infty \})
\end{align*}
(the modulus pair $N_1(01)$ is already defined in (\ref{defn-n1-mod}) and (\ref{defn-n1-pair})).
These modulus pairs obviously form a square $\ul{N}_1 \in \mathbf{MCor}^{\mathbf{Sq}}$, and the identity maps on the total spaces induce an admissible morphism 
\[
\ul{N}_1 \to \ul{N}_0.
\]

Indeed, the existence of this map follows from the definition of $\ul{N}_1$ and the minimality of the square $\ul{N}_0$ (see Definition \ref{defn-minimal-sq} for the definition of the minimality of squares).

Moreover, for each $(i,j)$, the minimal morphism $S(ij) \to N_0(ij)$ lifts to an admissible morphism $S(ij) \to N_1(ij)$, which is automatically minimal.
Indeed, this is trivial by definition for $(i,j)=(1,1), (1,0)$.
The minimality for $(i,j)=(0,1)$ follows from Lemma \ref{l-w}.
Finally, the minimality for $(i,j)=(0,0)$ follows from
\begin{align*}
N_1(00)^\infty |_{\ol{S}(00)} &= \sup \{v_l^\ast N_0(10)^\infty , h_u^\ast N_1(01)^\infty \}|_{\ol{S}(00)} \\
&= \sup \{v_l^\ast N_0(10)^\infty |_{\ol{S}(00)} , h_u^\ast N_1(01)^\infty |_{\ol{S}(00)} \} \\
%&= \sup \{S(00)^\infty , (N(01)^\infty + mW) |_{\ol{S}(00)} \} \\
&= S(00)^\infty ,
\end{align*}
where the last equality follows from Lemma \ref{l-b-1}.
Therefore, we conclude that
\[
\ul{N}_1 \in \Comp_1(\ul{S}).
\]

\begin{prop}\label{eq6.9} The square $\ul{N}_1$ has the following properties:
\begin{enumerate}
\item It is universally minimal.
\item We have
\[
v_l^\ast N_1(10)^\infty |_{\ol{N}_1(00) \setminus \tilde{Z}(00)} \leq h_u^\ast N_1(01)^\infty |_{\ol{N}_1(00) \setminus \tilde{Z}(00)} .
\]
(See \eqref{z-tilde} for the definition of $\tilde{Z}(00)$.)
\end{enumerate}
\end{prop}

\begin{proof}
(1) is obtained by combining \eqref{asumption6.6}, the definition of $N_1(00)$ and Lemma \ref{sup-lemma-pre}.
(2) is immediate by Lemma \ref{l-m} and \eqref{defn-n1-pair}.
\end{proof}

To prove Theorem \ref{existence-MV}, it suffices to show the following Theorem.

\begin{thm}\label{MV-curve}
Any square $\ul{N}_1 \in \Comp_1(\ul{S})$ having the properties of Proposition \ref{eq6.9} is an MV-square.
\end{thm}

The rest of this section will be devoted to the proof of Theorem \ref{MV-curve}.

%{\color{red}
%\begin{rk}\label{rk6.9}
%In the proof of Theorem \ref{MV-curve}, we only use the properties in Proposition \ref{eq6.9} of the square $\ul{N}_1 \in \Comp_1(\ul{S})$.
%In other words, the properties in Proposition \ref{eq6.9} provide a sufficient condition for $\ul{N}_1 \in \Comp_1 (\ul{S})$ to be MV.
%\end{rk}
%}

\subsection{Key Proposition}\label{subsec-resurgence}

%We prove that the square $\ul{N}_1$, constructed in the previous subsection, is MV.
Take any modulus pair $M=(\ol{M},M^\infty )$.
The key step is to prove the following proposition. The proof of Theorem \ref{MV-curve} will be finished in \S \ref{pf-mv-curve}.

\begin{prop}[Resurgence principle]\label{l-resurgence}
Let $\alpha_1$ and $\alpha_2$ be two \emph{distinct} elementary finite correspondences in $\ulMCor(M , N_1(01))$.
%Let $\xi_1$ and $\xi_2$ be the generic points of $\alpha_1$ and $\alpha_2$, respectively.
Denote by $\rho$ the morphsim $1_{M^\o}\times v_r^\o:M^\circ \times N_1(01)^\circ \xrightarrow{} M^\circ \times N_1(11)^\circ$.
Assume that the equality of sets 
\[
\rho (\alpha _1 ) = \rho (\alpha_2 ) =: \beta
\] 
holds.
%Define $\beta := \ol{\{\eta\}}$ to be the closure in $M^\circ \times N_1(11)^\circ$ (given the reduced scheme structure).
%It is easy to see that $\beta$ is an elementary finite correspondence in $\Cor (M^\circ , N_1(11)^\circ )$, and coincides with the support of the image of $\alpha_i$ under the map $\Cor (M^\circ , N(01)^\circ ) \to \Cor (M^\circ , N(11)^\circ )$ for each $i=1,2$.
Note that by Lemma \ref{elementary-composition}, $\beta$ is an elementary finite correspondence in $\Cor (M^\circ , N_1(11)^\circ )$, and coincides with the support of the image of $\alpha_i$ in $\Cor (M^\circ , N_1(11)^\circ )$ for $i=1,2$.

Then, we have
\begin{align*}
\beta &\in \Cor (M^\circ , N_1(10)^\circ ) \subset \Cor (M^\circ , N_1(11)^\circ ), \\
\alpha_1 , \alpha_2 &\in \Cor (M^\circ , N_1(00)^\circ ) \subset \Cor (M^\circ , N_1(01)^\circ ),
\end{align*}
and moroever
\begin{align*}
\beta &\in \ulMCor(M , N_1(10)), \\
\alpha_1 ,\alpha_2 &\in \ulMCor (M , N_1(00)).
\end{align*}
\end{prop}

%\begin{proof}[Proof of Proposition \ref{l-resurgence}]

%In the following, for the simplicity of notation, we set 
%\[\ol{N}(ij) := \ol{N}_1(ij) ,\  N(ij)^\infty := N_1(ij),\  \ul{N}:=\ul{N}_1.\]

The proof of Proposition \ref{l-resurgence} will be divided into several steps.

\subsubsection{The separation lemma}

Let $\ol{\beta}$ be the closure of $\beta $ in $\ol{M} \times \ol{N}_1(11)$, and $\ol{\beta}^N$ its normalization.
Similarly, let $\ol{\alpha}_i$ be the closure of $\alpha_i$ in $\ol{M} \times \ol{N}_1(01)$, and $\ol{\alpha}_i^N$ its normalization.
The proper morphism $\ol{M} \times \ol{N}_1(01) \to \ol{M} \times \ol{N}_1(11)$ induces proper surjective morphisms $\ol{\alpha}_i \to \ol{\beta}$, and the universality of normalization induces proper surjective morphisms $\ol{\alpha}_i^N \to \ol{\beta}^N$.

Consider the morphisms
\[
\ol{\alpha}_i \times_{\ol{N}_1(01)} \ol{S}(01) \xrightarrow{\iota_i} \ol{\beta} \times_{\ol{N}_1(11)} \ol{S}(01)\quad (i=1,2).\]

Note that $\iota_i$ is closed immersion, because both morphisms $\ol{\alpha}_i \times_{\ol{N}_1(01)} \ol{S}(01)\to \ol{M}\times \ol{S}(01)$ and $\ol{\beta} \times_{\ol{N}_1(11)} \ol{S}(01)\to \ol{M}\times \ol{S}(01)$ are closed immersions. We have more precisely:

\begin{lemma}\label{cl6.1}
For each $i=1,2$, $\iota_i$ is the closed immersion of an irreducible component. These two irreducible components are distinct.
\end{lemma}

\begin{proof} We note that $\ol{\alpha}_i \times_{\ol{N}_1(01)} \ol{S}(01)$ contains $\alpha_i$, which is contained in $M^\o\times N_1(01)^\o=M^\o\times S(01)^\o$, and $\alpha_i$ is obviously dense in $\ol{\alpha}_i \times_{\ol{N}_1(01)} \ol{S}(01)$ which is therefore irreducible. This also shows that $\ol{\alpha}_i \times_{\ol{N}_1(01)} \ol{S}(01)$ ($i=1,2$) are distinct. To see that they are irreducible components,  we use the following equalities:
\[
\dim \ol{\alpha}_i \times_{\ol{N}_1(01)} \ol{S}(01) =^1 \dim \alpha_i =^2 \dim \beta =^3 \ol{\beta} \times_{\ol{N}_1(11)} \ol{S}(01),
\]
where $=^1$ is obvious, and the equalities $=^2$ and $=^3$ follow from the \'etaleness of the morphism $\ol{S}(01) \to \ol{N}_1(11)$.
\end{proof}

\begin{lemma}[Separation Lemma]\label{separation}
The composition
\[
%(\sqcup_{i=1,2} \ol{\alpha}_i^{N}) \times_{\ol{N}_0(01)} Z(00) \to \sqcup_{i=1,2} \ol{\alpha}_i^{N} \to \ol{\beta}^N \text{ for $i=1,2$}
(\ol{\alpha}_1^N\sqcup \ol{\alpha}_2^N) \times_{\ol{N}_0(01)} Z(00) \to \ol{\alpha}_1^N\sqcup \ol{\alpha}_2^N \to \ol{\beta}^N
\]
is a closed immersion.%, where $\sqcup_{i=1,2} \ol{\alpha}_i^{N}$ denotes the disjoint union of the schemes $\ol{\alpha}_1^N$ and $\ol{\alpha}_2^N$.
\end{lemma}

\begin{proof}
Consider the following commutative diagram (recall that $S(11)=N_1(11)$):
\[\begin{xymatrix}{
(\ol{\alpha}_i \times_{\ol{N}_1(01)} \ol{S}(01))^N \ar[r] \ar[d]_{\tilde{\iota}_i} & \ol{\alpha}_i \times_{\ol{N}_1(01)} \ol{S}(01) \ar[d]_{\iota_i} & \\
(\ol{\beta} \times_{\ol{N}_1(11)} \ol{S}(01))^N \ar[r] \ar[d]_{\et}  \ar@{}[rd]|\square & \ol{\beta} \times_{\ol{N}_1(11)} \ol{S}(01) \ar[r] \ar[d]_{\et} \ar@{}[rd]|\square & \ol{S}(01) \ar[d]_{\et} \\ 
\ol{\beta}^N \ar[r] & \ol{\beta} \ar[r] & \ol{N}_1(11)
}\end{xymatrix}\]
where the south-west square is cartesian by the fact that normalization is compatible with \'etale  base change (see Lemma \ref{lB.1}). 
For the existence of the north-west square, see \cite[Cor. 6.3.8]{EGA2} and Lemma \ref{cl6.1}. %\marginpar{\color{red} We have to find an appropriate reference of this fact, or add a proof.}\footnote{Stacks Project: Tag 07TD, \url{https://stacks.math.columbia.edu/tag/07TD}}

Lemma \ref{cl6.1}  implies that $(\ol{\alpha}_1 \times_{\ol{N}_1(01)} \ol{S}(01))^N$ and $(\ol{\alpha}_2 \times_{\ol{N}_1(01)} \ol{S}(01))^N$ are distinct connected components of $(\ol{\beta} \times_{\ol{S}(11)} \ol{S}(01))^N$. %{\color{blue} Why are they distinct?}
In particular, the map
\begin{align*}
(\ol{\alpha}_1^N\sqcup \ol{\alpha}_2^N) \times_{\ol{N}_1(01)} \ol{S}(01) &= (\ol{\alpha}_1 \times_{\ol{N}_1(01)} \ol{S}(01))^N\sqcup (\ol{\alpha}_2 \times_{\ol{N}_1(01)} \ol{S}(01))^N\\
&\xrightarrow{\sqcup \tilde{\iota}_i} (\ol{\beta} \times_{\ol{S}(11)} \ol{S}(01))^N \\
&= \ol{\beta}^N \times_{\ol{S}(11)} \ol{S}(01)
\end{align*}
is a closed immersion. Taking the base change by the closed immersion $Z(00) \to \ol{S}(01)$, we conclude that the morphism 
\[
(\ol{\alpha}_1^N\sqcup \ol{\alpha}_2^N) \times_{\ol{N}_1(01)} Z(00) \to \ol{\beta}^N \times_{\ol{S}(11)} Z(00)
\]
is a closed immersion. 
Since $Z(00) \iso Z(10) \subset \ol{S}(11)$ is a closed immersion, the morphism
\[
\ol{\beta}^N \times_{\ol{S}(11)} Z(00) \to \ol{\beta}^N
\]
is a closed immersion. Therefore, the composite map
\[
(\ol{\alpha}_1^N\sqcup \ol{\alpha}_2^N) \times_{\ol{N}_1(01)} Z(00) \to \ol{\beta}^N \times_{\ol{S}(11)} Z(00) \to \ol{\beta}^N
\]
is a closed immersion. This finishes the proof of Lemma \ref{separation}.
\end{proof}

\subsubsection{Lifting of the closures}

\begin{lemma}
The generic point $\eta$ of $\beta$ lies in $M^\circ \times N_1(10)^\circ$.
Moreover, for each $i=1,2$, the generic point $\xi_i$ of $\alpha_i$ lies in $M^\circ \times N_1(00)^\circ$.
\end{lemma}

\begin{proof}
Note that $N_1(ij)^\circ = S(ij)^\circ$ for each $i,j \in \{0,1\}$.
First, we prove that $\eta \in M^\circ \times N_1(10)^\circ$.
For this, it suffices to prove that $\eta \notin M^\circ \times Z(10)$ since $N_1(11)^\circ \setminus N_1(10)^\circ \subset Z(10)$.
Suppose that $\eta$ lies in $M^\circ \times Z(10)$. 
Then, for each $i=1,2$, the generic point $\xi_i$ lies in $\alpha_i^N \times_{\ol{S}(11)} Z(10) \cong \alpha_i^N \times_{\ol{S}(01)} Z(00)$, and its image in $\beta^N$ is $\eta$.
However, this contradicts Lemma \ref{separation}.
Therefore, we conclude that $\eta \in M^\circ \times N_1(10)^\circ$.
Finally, since $N_1(00)^\circ \cong N_1(10)^\circ \times_{N_1(11)^\circ } N_1(01)^\circ$, and since $\xi_i$ is over $\eta$ for each $i=1,2$, we have $\xi_i \in M^\circ \times N_1(00)^\circ$.
This finishes the proof.
\end{proof}

Define closed subsets 
\begin{align*}
\tilde{\ol{\beta}} &:= \text{ the closure of $\{\eta\}$ in $\ol{M} \times \ol{N}_1(10)$}, \\
\tilde{\ol{\alpha}_i} &:= \text{ the closure of $\{\xi_i\}$ in $\ol{M} \times \ol{N}_1(00)$},
\end{align*}
and regard them as integral closed subschemes with the reduced scheme structures.
Then, the proper morphism $\ol{M} \times \ol{N}_1(10) \to \ol{M} \times \ol{N}_1(11)$ (resp. $\ol{M} \times \ol{N}_1(00) \to \ol{M} \times \ol{N}_1(01)$) induces a proper surjective morphism $\tilde{\ol{\beta}} \to \ol{\beta}$ (resp. proper surjective morphisms $\tilde{\ol{\alpha}}_i \to \ol{\alpha}_i$). 
Then, the universality of normalization induces a commutative diagram
\[\begin{xymatrix}{
\tilde{\ol{\alpha}}_i^N \ar[r] \ar[d] & \ol{\alpha}_i^N \ar[d] \\
\tilde{\ol{\beta}}^N \ar[r] & \ol{\beta}^N
}\end{xymatrix}\]
where all the arrows are proper surjective morphisms.

\subsubsection{Modulus condition on $\tilde{\ol{\beta}}^N$}

%In this step, we prove $\tilde{\ol{\beta}} \in \ulMCor (M,N(10))$.
Now, we prove the following crucial result.

\begin{prop}\label{prop-eq-aim}
The following inequality holds:
\begin{equation}\label{eq-aim}
\ol{M} \times N_1 (10)^\infty |_{\tilde{\ol{\beta}}^N} \leq M^\infty \times \ol{N}_1(10) |_{\tilde{\ol{\beta}}^N} .
\end{equation}
In particular, $ \tilde{\ol{\beta}}\cap M^\o \times N_1 (10)^\infty=\emptyset$.
\end{prop}

\begin{proof}
Note that the last assertion follows readily from \eqref{eq-aim}.

By the assumption that $\alpha_i \in \ulMCor (M,N_1(01))$, we have
\[
\ol{M} \times N_1(01)^\infty |_{\ol{\alpha}^{N}_i} \leq M^\infty \times \ol{N}_1(01) |_{\ol{\alpha}^{N}_i} .
\]
Pulling back by the morphism $\tilde{\ol{\alpha}}_i^N \to \ol{\alpha}^{N}_i$, we obtain
\begin{equation}\label{eq-assumption}
\ol{M} \times h_u^\ast N_1(01)^\infty |_{\tilde{\ol{\alpha}}_i^N} \leq M^\infty \times \ol{N}_1(00) |_{\tilde{\ol{\alpha}}_i^N}
\end{equation}
for each $i=1,2$. 

Set 
\begin{equation}\label{eq6.1}
\tilde{\ol{\alpha}}_i^\circ := \tilde{\ol{\alpha}}_i \setminus (\ol{M} \times \tilde{Z}(00))
\end{equation}
(see (\ref{z-tilde}) for the definition of $\tilde{Z}(00)$).
Then, we have
\begin{equation}\label{eq-open}
\ol{M} \times v_l^\ast N_1(10)^\infty  |_{\tilde{\ol{\alpha}}_i^{\circ N}} 
\leq^1 \ol{M} \times h_u^\ast N_1(01)^\infty  |_{\tilde{\ol{\alpha}}_i^{\circ N}}  
\leq^2 M^\infty \times \ol{N}_1(00) |_{\tilde{\ol{\alpha}}_i^{\circ N}} ,
\end{equation}
where $\tilde{\ol{\alpha}}_i^{\circ N}$ denotes the normalization of $\tilde{\ol{\alpha}}_i^{\circ}$, which is an open subscheme of $\tilde{\ol{\alpha}}_i^N$.
Here, the inequality $\leq^1$ follows from Proposition \ref{eq6.9}, and the inequality $\leq^2$ follows from (\ref{eq-assumption}).
%In the following, we deduce the inequality (\ref{eq-aim}) from (\ref{eq-open}).
%The key observation is the following Lemma \ref{separation-lift}.

\begin{lemma}\label{separation-lift}
The images of the proper morphisms 
\[
\tilde{\ol{\alpha}}_i^{N} \times_{\ol{N}_1(00)} \tilde{Z}(00) \to \tilde{\ol{\beta}}^N \text{ for $i=1,2$}
\]
do not intersect with each other. Here, recall that $\ol{N}_1(00) = \ol{N}_0(00)$ by definition. For the definition of $\tilde{Z}(00)$, see \eqref{z-tilde}.
\end{lemma}
\begin{proof}
Consider the commutative diagram
\[\begin{xymatrix}{
\tilde{\ol{\alpha}}_i^{N} \times_{\ol{N}_0(00)} \tilde{Z}(00) \ar[r] \ar[d]_{\tilde{\pi}_i} & \ol{\alpha}_i^{N} \times_{\ol{N}_0(01)} Z(00)  \ar[d]^{\pi_i} \\
\tilde{\ol{\beta}}^N \ar[r] & \ol{\beta}^N .
}\end{xymatrix}\]
By Lemma \ref{separation}, the images of the maps $\pi_1$ and $\pi_2$ do not intersect with each other.
Therefore, the commutativity of the diagram shows that the images of the maps $\tilde{\pi}_1$ and $\tilde{\pi}_2$ do not intersect with each other.
This finishes the proof of Lemma \ref{separation-lift}.
\end{proof}

We recall the following lemma from \cite[Lemma 2.2]{Krishna-Park}:
\begin{lemma}\label{covering-criterion}
Let $f : Y \to X$ be a surjective morphism between normal integral $k$-schemes. 
Let $D$ be a Cartier divisor on $X$. Then, $D$ is effective if and only if $f^\ast (D)$ is effective.\qed
\end{lemma}
%\begin{proof}
%The ``only if '' part is trivial. The proof for the ``if'' part is given in \cite[Lemma 2.2]{Krishna-Park}.
%\end{proof}

%\marginpar{\color{red} To apply Lemma \ref{covering-criterion}, we need this argument. The surjectivity of $\sqcup_{i=1,2} \ol{\alpha}_i^{\circ N} \to \ol{\beta}^N$ does not seem to be enough.}

For each $i=1,2$, define 
\[
B_i := \text{ the image of the map $\tilde{\ol{\alpha}}_i^{N} \times_{\ol{N}_1(00)} \tilde{Z}(00) \to \tilde{\ol{\beta}}^N$}.
\]
Then, by Lemma \ref{separation-lift}, $B_1$ and $B_2$ are disjoint closed subsets of $\tilde{\ol{\beta}}^N$.
Therefore, it suffices to prove the inequality (\ref{eq-aim}) on each open subset $\tilde{\ol{\beta}}^N \setminus B_1$ and $\tilde{\ol{\beta}}^N \setminus B_2$.
For each $i=1,2$, define a closed subset $A_i \subset \tilde{\ol{\alpha}}_i^N$ as the fiber of $B_i$ by the map $\tilde{\ol{\alpha}}_i^N \to \tilde{\ol{\beta}}^N$.
Since the induced morphisms
\[
\tilde{\ol{\alpha}}_i^N \setminus A_i \to \tilde{\ol{\beta}}^N \setminus B_i
\]
are surjective morphisms between normal integral $k$-schemes, Lemma \ref{covering-criterion} reduces the proof of the inequality (\ref{eq-aim}) to proving the following inequalities for $i=1,2$:
\begin{equation}\label{eq-reduced}
\ol{M} \times v_l^\ast N_1(10)^\infty  |_{\tilde{\ol{\alpha}}_i^N \setminus A_i} \leq M^\infty \times \ol{N}_1(00) |_{\tilde{\ol{\alpha}}_i^N \setminus A_i} .
\end{equation}
Since by definition we have 
\[
\tilde{\ol{\alpha}}_i^{N} \times_{\ol{N}_1(00)} \tilde{Z}(00) \subset A_i ,
\]
we obtain
\[
\tilde{\ol{\alpha}}_i^N \setminus A_i \subset \tilde{\ol{\alpha}}_i^{\circ N}
\]
where $\tilde{\ol{\alpha}}_i^\o$ we defined in \eqref{eq6.1}. Therefore, the inequality (\ref{eq-reduced}) immediately follows from (\ref{eq-open}).
This finishes the proof of Proposition \ref{prop-eq-aim}.
\end{proof}
 
\subsubsection{End of Proof of Proposition \ref{l-resurgence}}

We are reduced to proving the following Claim \ref{claim-beta} and Claim \ref{claim-alpha}.

\begin{claim}\label{claim-beta}
We have \[\beta \in \ulMCor (M,N_1(10)).\]
\end{claim}
\begin{proof}
First, we have to check that $\beta \in \Cor (M^\circ , N_1(10)^\circ )$.
Since $\beta$ is proper over $M^\circ$, we have $\beta = \ol{\beta} \times_{\ol{M}} M^\circ$.
Since $\tilde{\ol{\beta}} \to \ol{\beta}$ is surjective, so is the induced map $\tilde{\ol{\beta}} \times_{\ol{M}} M^\circ \to \ol{\beta} \times_{\ol{M}} M^\circ$.
By Proposition \ref{prop-eq-aim}, we have
\[
\tilde{\ol{\beta}} \times_{\ol{M}} M^\circ = \tilde{\ol{\beta}}\cap (M^\o\times \ol{N}_1(10)) \subset M^\circ \times N_1(10)^\circ ,
\]
which implies that
\[
\beta = \ol{\beta} \times_{\ol{M}} M^\circ \subset M^\circ \times N_1(10)^\circ .
\]
Therefore, we have $\beta \in \Cor (M^\circ , N_1(10)^\circ )$.
Then, the closure of $\beta$ in $\ol{M} \times \ol{N}_1(10)$ is (by definition) equal to $\tilde{\ol{\beta}}$, and the inequality (\ref{eq-aim}) shows that $\beta \in \ulMCor (M,N_1(10))$.
This finishes the proof of Claim \ref{claim-beta}.
\end{proof}

\begin{claim}\label{claim-alpha}
We have
\[\alpha_1 ,\alpha_2 \in \ulMCor (M , N_1(00)).\]
\end{claim}

\begin{proof}
Since we have canonical morphisms $\tilde{\ol{\alpha}}_i^N \to \tilde{\ol{\beta}}^N$, the inequality (\ref{eq-aim}) implies 
\begin{equation}\label{eq6.17}
\ol{M} \times v_l^\ast N_1(10)^\infty |_{\tilde{\ol{\alpha}}_i^N} \leq M^\infty \times \ol{N}_1(00) |_{\tilde{\ol{\alpha}}_i^N} 
\end{equation}
for each $i=1,2$.
By Proposition \ref{eq6.9} (1), the square $\ul{N}_1$ is universally minimal. Therefore, by combining \eqref{eq-assumption} and \eqref{eq6.17}, we obtain
\begin{equation}\label{eq-alpha}
\ol{M} \times N_1(00)^\infty |_{\tilde{\ol{\alpha}}_i^N} \leq M^\infty \times \ol{N}_1(00) |_{\tilde{\ol{\alpha}}_i^N} .
\end{equation}

In particular, $\tilde{\ol{\alpha}}_i \times_{\ol{M}} M^\circ = \tilde{\ol{\alpha}}_i \cap M^\circ \times \ol{N}_1(00)$ is contained in $M^\circ \times N_1(00)^\circ$.
Since the natural projection $\tilde{\ol{\alpha}}_i \to \ol{\alpha}_i$ is surjective, we have
\[
\alpha_i = \ol{\alpha}_i \times_{\ol{M}} M^\circ \subset M^\circ \times N_1(00)^\circ ,
\]
where the equality follows from the properness of $\alpha_i$ over $M^\circ$.
Therefore, we have $\alpha_i \in \Cor (M^\circ , N_1(00)^\circ )$.
The closure of $\alpha_i$ in $\ol{M} \times \ol{N}_1(00)$ is (by definition) equal to $\tilde{\ol{\alpha}}_i$.
Therefore, the inequality (\ref{eq-alpha}) shows that $\alpha_i \in \ulMCor (M,N(00))$ for each $i=1,2$.
This finishes the proof of Claim \ref{claim-alpha}.
\end{proof}

Thus, we have finished the proof of Proposition \ref{l-resurgence}.
%\end{proof}

\subsection{End of Proof of Theorem \ref{MV-curve}}\label{pf-mv-curve}

Let $M \in \ulMCor$ be a modulus pair.
By remark \ref{mv-sequence}, it suffices to prove that the sequence
\begin{equation*}
\Z_\tr N_1(00)(M) \to \Z_\tr N_1(10)(M) \oplus \Z_\tr N_1(01)(M) \to \Z_\tr N_1(11)(M)
\end{equation*}
is exact at the middle term.
Let $\alpha$ be an element of $\ulMCor(M ,N_1(01))$, and write 
\[
\alpha = \sum_{i\in I} m_i \alpha_i ,
\]
where $I$ is a finite set, $\alpha_i$ are prime cycles and $m_i$ are non-zero integers.
Denote by $\rho$ the map $\ulMCor(M ,N_1(01)) \to \ulMCor(M ,N_1(11))$.
Define 
\begin{align*}
I_1 &:= \{i \in I | \exists j \in I \setminus \{i\}, |\rho (\alpha_i )| = |\rho (\alpha_j )| \}, \\
I_2 &:= I \setminus I_1 .
\end{align*}
Then,  Proposition \ref{l-resurgence} implies that 
\begin{equation}\label{eq-I2}
\alpha_i  \in  \ulMCor(M,N_1(00)) \text{ for each $i \in I_1$.}
\end{equation}

Assume that 
\[
\rho (\alpha) \in \ulMCor(M ,N_1(10)) \subset \ulMCor(M ,N_1(11)).
\]
By (\ref{eq-I2}), we are reduced to showing the following claim.
\begin{claim}\label{claim-final}
For any $i \in I_2$, we have
\[
\alpha_i \in \ulMCor(M,N_1(00)).
\]
\end{claim}

\begin{proof}[Proof of Claim]
Take any $i \in I_2$.
By definition of $I_2$, the coefficient of the integral cycle $|h(\alpha_i)|$ in the cycle $\rho (\alpha )$ is non-zero.
Therefore, we have
\begin{equation}\label{eq-h-1}
\beta := |\rho(\alpha_i)| \in \ulMCor(M,N_1(10)),
\end{equation}
where $|\rho(\alpha_i)|$ denotes the irreducible support of the divisor.
In particular, we have
\begin{align*}
\beta &\in \Cor (M^\circ , N_1(10)^\circ), \\
\alpha_i &\in \Cor (M^\circ , N_1(00)^\circ ),
\end{align*}
where the second claim follows from $N_1(00)^\circ = N_1(10)^\circ  \times_{N_1(11)^\circ} N_1(01)^\circ$.

Let $\ol{\alpha}_i$ be the closure in $\ol{M} \times \ol{N}_1(00) $ and $\ol{\alpha}_i^N$ its normalization.
Similarly, let $\ol{\beta}$ be the closure of $\beta$ in $\ol{M} \times \ol{N}_1(10) $ and $\ol{\beta}^N$ its normalization.
Then, (\ref{eq-h-1}) implies that
\begin{equation}\label{eq-h-1a}
\ol{M} \times N_1(10)^\infty |_{\ol{\beta}^N} \leq M^\infty \times \ol{N}_1(10) |_{\ol{\beta}^N} .
\end{equation}
Since the maps $\ol{\alpha}_i \to \ol{\beta}$ induced by $\mathrm{id}_{\ol{M}} \times v_l : \ol{M} \times \ol{N}_1(00) \to \ol{M} \times \ol{N}_1(10)$ are dominant, the universality of normalization induces a morphism $\ol{\alpha}_i^N \to \ol{\beta}^N$.
And (\ref{eq-h-1a}) implies 
\begin{equation}\label{eq-h-1b}
\ol{M} \times v_l^\ast N_1(10)^\infty |_{\ol{\alpha}_i^N} \leq M^\infty \times \ol{N}_1(00) |_{\ol{\alpha}_i^N} .
\end{equation}
On the other hand, the assumption that $\alpha_i \in \ulMCor(M ,N_1(01))$ implies that
\begin{equation*}
\ol{M} \times N_1(01)^\infty |_{\ol{\alpha}_i^{\prime N}} \leq M^\infty \times \ol{N}_1(01) |_{\ol{\alpha}_i^{\prime N}} ,
\end{equation*}
where $\ol{\alpha}_i^{\prime }$ denotes the closure of $\alpha_i$ in $\ol{M} \times \ol{N}_1(01)$, and $\ol{\alpha}_i^{\prime N}$ its normalization. 
Since $\mathrm{id}_{\ol{M}} \times h_u : \ol{M} \times \ol{N}_1(00) \to \ol{M} \times \ol{N}_1(01)$ and the universality of normalization induce a canonical morphism $\ol{\alpha}_i^N \to \ol{\alpha}_i^{\prime N}$, we obtain
\begin{equation}\label{eq-h-2a}
\ol{M} \times h_u^\ast N_1(01)^\infty |_{\ol{\alpha}_i^{N}} \leq M^\infty \times \ol{N}_1(00) |_{\ol{\alpha}_i^{N}} .
\end{equation}
Since $N_1(00)^\infty = \sup \{v_l^\ast N_1(10)^\infty ,h_u^\ast N_1(01)^\infty \}$ by definition, the inequalities (\ref{eq-h-1b}) and (\ref{eq-h-2a}) imply
\[
\ol{M} \times N_1(00)^\infty |_{\ol{\alpha}_i^N} \leq M^\infty \times \ol{N}_1(00) |_{\ol{\alpha}_i^N} .
\]
Therefore, we have $\alpha_i \in \ulMCor(M,N_1(00))$.
This finishes the proof of Claim \ref{claim-final}.
\end{proof}

Thus, we finished the proof of Theorem \ref{MV-curve}. Therefore, we have proven Theorem \ref{existence-MV}.

\appendix

\section{Some remarks on the sup of Cartier divisors}

\subsection{Preliminary}

%Before stating the main results, we prepare an elementary lemma.

\begin{lemma}\label{key-lem}
Let $X$ be a scheme.
Suppose given three effective Cartier divisors $D_1, D_2$ and $E$ on $X$ such that $E \leq D_i$ for each $i=1,2$.

Then, we have:
\[
\text{$E = D_1 \times_X D_2$ \ \ iff. \ \ $|D_1 - E| \cap |D_2 - E| = \emptyset $.}
\]
\end{lemma}

\begin{remark}
The ``inf'' of two effective Cartier divisors might be zero even if $|D_1| \cap |D_2| \neq \emptyset$: for example, consider the case $X = \mathbb{A}^2 = \Spec (k[x_1,x_2])$ and $D_i = \{x_i = 0\}$.
\end{remark}

\begin{proof}[Proof of Lemma \ref{key-lem}]
%The inclusions $E \subset D_1$ and $E \subset D_2$ imply that $D_1 - E$ and $D_2 - E$ are effective.
Regard the effective Cartier divisors $D_i - E$ as closed subschemes on $X$, and set
\[
Z := (D_1 - E) \times_X (D_2 - E).
\]
%Then, we have the natural closed immersions
%\[
%Z \to D_1 - E , \ \ Z \to D_2 - E .
%\]
For a closed subscheme $i : V \to X$, we set \[I_V := \mathrm{Ker} (\mathcal{O}_X \to i_\ast \mathcal{O}_V ) . \]
Then, we have
\[
I_{D_1 \times_X D_2} = I_{D_1} + I_{D_2}.
\]
Since $I_{Z} = I_{D_1 - E} + I_{D_2 - E} = I_{D_1} \cdot I_E^{-1} + I_{D_2} \cdot I_E^{-1}$, we have
\[
I_{Z} \cdot I_E = (I_{D_1} \cdot I_E^{-1} + I_{D_2} \cdot I_E^{-1}) \cdot I_E = I_{D_1} + I_{D_2} ,
\]
where $I_{E}^{-1}$ denotes the inverse of the invertible ideal sheaf $I_E$.
Combining the above equalities, we obtain
\begin{equation}\label{eq-interpret}
I_{Z} \cdot I_E = I_{D_1 \times_X D_2} .
\end{equation}
Therefore, we have
\begin{align*}
|D_1 - E| \cap |D_2 - E| = \emptyset  &\Leftrightarrow Z = \emptyset \Leftrightarrow I_Z = \mathcal{O}_X \Leftrightarrow^\dag I_{D_1 \times_X D_2}  = I_E \\
&\Leftrightarrow  D_1 \times_X D_2 = E,
\end{align*}
where $\Leftrightarrow^\dag$ follows from (\ref{eq-interpret}) and the fact that $I_E$ is invertible.
This finishes the proof of Lemma \ref{key-lem}.
\end{proof}

\subsection{Compatibility between pullback and $\sup$, after normalized blow-up}

For two Weil divisors $Z_1 $ and $Z_2$ on a scheme $X$, denote by \[\sup\nolimits^{\Wei} (Z_1,Z_2)\] the smallest Weil divisor on $X$ which is larger than of equal to $Z_1$ and $Z_2$.
For two Cartier divisors $D_1$ and $D_2$ on a scheme $X$, denote by \[\sup\nolimits^{\Car} (D_1,D_2)\] the smallest Cartier divisor on $X$ which is larger than of equal to $D_1$ and $D_2$ (if it exists).

\begin{lemma}\label{sup-lemma-pre}
Let $X$ be a normal scheme, and let $D_1, D_2$ and $E$ be effective Cartier divisors on $X$ such that $E \leq D_i$ for each $i=1,2$.
Assume one of the following equivalent conditions:
\begin{itemize}
\item[(a)] $E = D_1 \times_X D_2$.
\item[(b)] $|D_1 - E| \cap |D_2 - E| = \emptyset$.
\end{itemize}
(The equivalence of these conditions follow from Lemma \ref{key-lem}.)

Then, the following assertions hold:
\begin{itemize}
\item[(1)] Regard $D_1$ and $D_2$ as Weil divisors on $X$ (since $X$ is normal, any Cartier divisor can be naturally regarded as a Weil divisor).
Then, we have 
\[
\sup\nolimits^{\Wei} (D_1,D_2) = D_1 + D_2 - E.
\]
\item[(2)] $\sup\nolimits^{\Wei} (D_1,D_2)$ is an effective Cartier divisor, and is equal to the sup as Cartier divisor:
\[
\sup\nolimits^{\Wei} (D_1,D_2) = \sup\nolimits^{\Car} (D_1,D_2) =: \sup (D_1,D_2) .
\]
\item[(3)] The square 
\[\begin{xymatrix}{
(X,\sup(D_1 ,D_2 )) \ar[r] \ar[d]& \ar[d](X,  D_1) \\
(X, D_2 ) \ar[r]& (X, E )
}\end{xymatrix}\]
is universally minimal in the sense of \cite[Def. 4.3.9]{motmod}. In other words:
let $f : Y \to X$ be a morphism such that $Y$ is normal, and such that the pullback of Cartier divisors $D_1, D_2$ and $E$ are well-defined.
Then, we have
\[
f^\ast \sup (D_1,D_2) = \sup (f^\ast D_1,f^\ast D_2).
\]
(See the assertion (2) for the definition of $\sup$.)
\end{itemize}
\end{lemma}

\begin{proof}
%Let $*\in \{\Wei,\Car\}$. 
We may calculate $\sup\nolimits^\Wei (D_1,D_2)$ as follows:
\begin{align*}
\sup\nolimits^\Wei (D_1,D_2) &=^1 \sup\nolimits^\Wei (D_1 - E,D_2 - E) + E \\
&=^2 (D_1 - E) + (D_2 - E) + E \\
&= D_1 + D_2 - E,
\end{align*}
where $=^1$ is obvious, and $=^2$ follows from the assumption.
This proves the assetion (1).

Since the right hand side of $\sup\nolimits^\Wei (D_1,D_2) = D_1 + D_2 - E$ is Cartier, the assertion (2) is obvious.

Finally, we prove the assertion (3). Let $f : Y \to X$ be a morphism as in the statement. 
Note that the scheme
\begin{equation}\label{eq-cartier-01}
f^\ast D_1 \times_Y f^\ast D_2 = (D_1 \times_X Y) \times_Y (D_2 \times_X Y) = D_1 \times_X D_2 \times_X Y = f^\ast (D_1 \times_X D_2).
\end{equation}
is an effective Cartier divisor on $Y$.
Then, we have
\begin{align*}
\sup\nolimits^{\Wei} (f^\ast D_1,f^\ast D_2) &=^1 f^\ast D_1 + f^\ast D_2 - (f^\ast D_1 \times_Y f^\ast D_2) \\
&=^2 f^\ast D_1 + f^\ast D_2 - f^\ast (D_1 \times_X D_2) \\
&=^3 f^\ast (D_1 + D_2 - D_1 \times_X D_2) \\
&=^4 f^\ast \sup\nolimits^{\Wei} (D_1,D_2),
\end{align*}
where $=^1$ and $=^4$ follow from the assertion (1), $=^2$ follows from (\ref{eq-cartier-01}), and $=^3$ follows from the compatibility between the addition and the pullback of Cartier divisors.
This finishes the proof of Lemma \ref{sup-lemma-pre}.
\end{proof}

\begin{prop} \label{pA.1} Let $X\in \Sch$, and let $D_1, D_2$ be effective Cartier divisors on $X$. Assume that $X \setminus |D_i|$ are smooth. Set $E := D_1 \times_X D_2$. Let
$\pi:X' :=(\Bl_EX)^N \to \Bl_EX\to X$ be the normalized blow-up along $E$. Set
\[D_i' =D_i \times_X X',\quad E' :=E\times_X X'.\]
Then, the pair $(X',E')$ forms a modulus pair, $S:=\sup^\Car(D'_1,D'_2)$ exists, we have
\begin{align*}
( X' , D'_ i ) &\iso ( X  , D_i  ) ,\\
(X', S)^\o &= X \setminus (|D_1| \cup |D_2|), \\
(X', E')^\o &= X \setminus (|D_1| \cap |D_2|).
\end{align*}
and the square
\begin{equation}\label{eqA.1}
\begin{CD}
(X', S) @>>>   (X', D_1' )\\
@VVV @VVV\\
( X ' , D_2' ) @>>>   ( X ' , E ' )
\end{CD}
\end{equation}
is universally minimal and extends the elementary Zariski square
\[\begin{CD}
(X' \setminus  (|\tilde D_1' | \cup |\tilde D_2' |, E'|_{X' \setminus (|\tilde D_1' | \cup |\tilde D_2' |})@>>> (X' \setminus |\tilde D_1' |, E'|_{X' \setminus |\tilde D_1' }|)\\
@VVV @VVV\\
(X' \setminus |\tilde D_2' |, E'|_{X' \setminus |\tilde D_2' |})@>>> (X', E'),
\end{CD}\]
where we set $\tilde D_i' := D_i' \setminus E'$.
\end{prop}

\begin{proof} Note that $X \setminus E = (X \setminus |D_1|) \cup (X \setminus |D_2|)$ is smooth. Therefore, $(X,E)$ is a modulus pair. The existence of $\sup^\Car(D'_1,D'_2)$ follows from Lemma \ref{sup-lemma-pre} (2). The next claims are easy. By construction, we have
\[E ' = D_1' \times_{X '} D_2',\] 
and $E'$ is an effective Cartier divisor. Therefore the statement about \eqref{eqA.1} follows from Lemma \ref{sup-lemma-pre} (3).  Note that the second square is indeed elementary Zariski, since $|\tilde D_1' | \cap |\tilde D_2' | = \emptyset$. %Therefore, \cite[Proposition 4.3.10]{motmod} shows that the sequence in the statement is exact in MNST. 
The normality of $X \setminus E$ implies that the birational morphism $\pi$ is an isomorphism over $X\setminus E$. Therefore we obtain the second assertion. This finishes the proof.
\end{proof}

\section{Elementary Lemmas}

\subsection{Modulus increasing lemma}

We recall the following lemma from \cite[Lemma 3.16]{cubeinv}.

\begin{lemma}\label{m-i-l}
Let $X$ be a quasi-compact scheme and let $D, E$ be Cartier divisors on $X$ with $E \geq 0$.
Assume that the restriction of $D$ to the open subset $X \setminus E \subset X$ is effective.
Then, there exists a natural number $n_0 \geq 1$ such that $D + n \cdot E$ is effective for any $n \geq n_0$. \qed
\end{lemma}

\subsection{No extra fiber lemma}\label{section-no-extra-fiber}

\begin{lemma}[No extra fiber lemma]\label{no-extra-fiber}
Let 
\[\begin{xymatrix}{
 & Y \ar[d]^p \\
  U\ar[ru]^f \ar[r]_j & X
}\end{xymatrix}\]
be a commutative triangle of schemes. We assume:
\begin{itemize}
\item $p$ is separated;
\item $f$ is scheme-theoretically dominant \cite[Ch. 1, Def. 5.4.2]{EGA1}.
\end{itemize}
Let $V=U \times_{X} Y$. Then $p': V \to U$ is an isomorphism.
If $j$ is an open immersion, so is $f$.
\end{lemma}

\begin{remark}
The morphism $f : U \to Y$ is scheme-theoretically dominant if $Y$ is reduced and $f$ is dominant \cite[Ch. I, Prop. 5.4.3]{EGA1}.
\end{remark}

\begin{proof}
The diagram yields a section $s : U \to V$ of $p'$, and we need to show that $sp'$ is the identity of $V$. Since $f$ is scheme-theoretically dominant, so is $s$; since $p$ is separated, so is $p'$. Then it suffices to show that $sp'$ and $1_V$ agree after composition with $s$, which is obvious.
The last assertion follows from the fact that $V \to Y$ is an open immersion if $j$ is.
%\[\begin{xymatrix}{
%U \ar[rd]_{=} \ar[r]^(0.4){s} & p^{-1}(U) \ar[r] \ar@{}[rd]|\square \ar[d]^{\mathrm{sep.}}_{p'} & X \ar[d]_p^{\mathrm{sep.}} \\
%&  U \ar[r]_j & Y.
%}\end{xymatrix}\]
\end{proof}

%\begin{lemma}
%Let $f : X \to Y$ be a separated morphism of schemes, and let $U \subset X$ be an open dense subset. 
%Assume that the image $f(U)$ of $U$ is open in $Y$, and the induced morphism $U \to f(U)$ is proper (e.g., an isomorphism).
%Then, we have $f^{-1}(f(U)) = U$.\end{lemma}

%\begin{proof}
%Consider the commutative diagram
%\[\begin{xymatrix}{
%U \ar[rd]_{\mathrm{proper}} \ar[r]^(0.38)j & f^{-1}(f(U)) \ar[r] \ar@{}[rd]|\square \ar[d]^{\mathrm{sep.}} & X \ar[d]^{\mathrm{sep.}} \\
%&  f(U) \ar[r] & Y}\end{xymatrix}\]
%where all the horizontal arrows are open immersions, the square is cartesian and the two vertical morphisms are separated.
%The triangle diagram on the left implies that the dense open immersion $j$ is proper, which proves that $j$ is an isomorphism.
%This finishes the proof.\end{proof}

\subsection{Normalization and smooth base change}

\begin{lemma}\label{lB.1} Let $f:Y\to X$ be a smooth morphism, where $X,Y\in \Sch(k)$, and let $X^N\to X$ be the normalization of $X$. Then, $f^N:Y\times_X X^N\to Y$ is the normalization of $Y$.
\end{lemma}

\begin{proof} Since $Y\times_X X^N\to X^N$ is smooth, $Y\times_X X^N$ is normal; since $f^N$ is dominant, it factors through the normalization $Y^N\to Y$. But $f^N$ is finite since $f$ is, hence the morphism $Y\times_X X^N\to Y^N$ is an isomorphism.
\end{proof}

\subsection{Proof of Lemma \ref{strong-lemma}}\label{pf-strong-lemma}

%For our application, we may assume that $f$ is quasi-finite and of finite presentation.
%So, we also assume these.
%The reduction of the general case to this situation is omitted.

First, consider the case that $f$ is finite. Then, the proof is easy. Indeed, since $f$ is finite flat of finite presentation, it is finite locally free.
Since $f$ is an isomorphism over a dense open subset, the rank of $f$ is equal to $1$, which implies that $f$ is an isomorphism.

Next, consider the case that $f$ is quasi-finite, flat, separated of finite presentation.
Since $f$ is flat, the image $f(X) \subset S$ is open. Therefore, we may assume that $f$ is surjective.
It suffices to prove that $f$ is finite.

\ 

We need the following propositions from \cite[\S 2.3 Prop. 8 (a), \S 2.5 Prop. 2]{neronmod}:
\begin{prop}[\'etale localization of quasi-finite morphisms]\label{et-loc-of-qfin}
Let $f : X \to Y$ be locally of finite type. Let $x$ be a point of $X$, and set $y:=f(x)$.

If $f$ is quasi-finite at $x$, then there exists an \'etale neighborhood $Y' \to Y$ of $y$ such that the morphism $f' : X' \to Y'$, obtained from $f$ by the base change $Y' \to Y$ induces a finite morphism $f' |_{U'} : U' \to Y'$, where $U'$ is an open neighborhood of the fiber of $X' \to X$ above $x$.
In addition, if $f$ is separated, $U'$ is a connected component of $X'$. \qed
\end{prop}

\begin{prop}[compatibility between schematic images and flat base changes]\label{sch-im-flat-bc}
Let $f : X \to Y$ be an $S$-morphism which is quasi-compact and quasi-separated.
Let $g : S' \to S$ be a flat morphism, and denote by $f':X'\to Y'$ be the $S'$-morphism obtained from $f$ by base change.
Let $Z$ (resp. $Z'$) be the schematic image of $f$ (resp. $f'$).
Then, $Z \times_S S'$ is canonically isomorphic to $Z'$. \qed
\end{prop}

Since the finiteness of $f$ is Zariski local on $S$, it suffices to check over an open neighborhood of a fixed point $s \in S$.
Take a point $x \in X$ above $s$.
Take an \'etale neighborhood $g : S' \to S$ of $s$ as in Proposition \ref{et-loc-of-qfin}, and set $X' := X \times_S S'$.
Denote by $f'$ the induced morphism $X' \to S'$.
Since $f$ is separated, quasi-finite and locally of finite type, there exists a connected component $V' \subset X'$ such that 
\[
f'|_{V'} : V' \to X' \to S'
\]
is finite, and $V'$ is an open neighborhood of the fiber of $x$. 
Since $f$ is flat, so is $f'$, hence the image $f'(V') \subset S'$ is an open subset.
By shrinking $S'$, we may assume that $V' \to S'$ is surjective.
Since $f$ is an isomorphism over $U \subset S$, $f'$ is an isomorphism over $g^{-1} (U) \subset S'$.
Therefore, combining with the surjectivity of $f'|_{V'}$, we have 
\begin{equation}\label{eq-a}
(f')^{-1} (g^{-1} (U)) \subset V' .
\end{equation}
On the other hand, since the map $g \circ f'$ is a flat morphism, Proposition \ref{sch-im-flat-bc} implies that the open subset
\begin{equation}\label{eq-b}
(f')^{-1} (g^{-1} (U)) \subset X' = V' \sqcup X'_1 ,
\end{equation}
is schematically dense, where $X'_1$ is an open and closed subset of $X'$.
Therefore,  (\ref{eq-a}) shows that $X'_1 = \emptyset$ and $X' = V'$.
Therefore, we have $f' = f' |_{V'}$, hence $f'$ is finite.

By replacing $S$ by the image of $S' \to S$, we may assume that $S' \to S$ is an fpqc-covering.
Since finiteness is an fpqc-local property, we conclude that $f$ is finite. This finishes the proof of Lemma \ref{strong-lemma}.

\section{Complements on pro-adjoints}\label{sB} We keep the notation of \cite[A.2]{motmod}.

\subsection{Canonical representation}

Let $u:\sC\to \sD$ have a pro-adjoint $v$: it is unique up to unique isomorphism of functors. For $d\in \sD$, we may write $v(d)= ``\lim"_{i\in I} c_i$ for some suitable inverse system $I\to \sC$. By \cite[Exp. I, Prop. 8.1.6]{SGA4}, we may choose $I$ to be a cofiltering ordered set. More specifically, let us pick once and for all, for each $d\in \sD$, a cofiltering ordered set $I(d)$ and a functor $\ul{c}:I(d)\to \sC$ such that $``\lim"_{i\in I(d)} c_i\in \pro{}\sC$ corepresents the functor $c'\mapsto \sD(d,u(c'))$. Then the pairs $(I(d),\ul{c})$ assemble to a functor isomorphic to $v$. The unit of the adjunction yields a system of compatible morphisms
\[(d\by{f_i} u(c_i))_{i\in I(d)}\]
which defines a subcategory $I'(d)$ of $d\downarrow u$, with an obvious functor $\phi:I(d)\to I'(d)$, and we have the following tautology:

\begin{lemma}\label{l0} $\phi$ is an isomorphism of categories.\qed
\end{lemma}

Thus we may identify $I'(d)$ with $I(d)$, thus view $I(d)$ as a subcategory of $d\downarrow u$, the functor $\ul{c}$ being given by the second projection.

\subsection{Categories of diagrams}\label{s1} Let $\Cat$ be the $2$-category of categories, and let $\Delta$ be a small category. %(More generally, $\Delta$ might be a ``partial category'', i.e. a diagram with commutation relations as in \cite[Appendix, Prop. 3.3]{am}.) 
We have a $2$-functor 
\begin{align*}
\Delta_*:\Cat&\to \Cat\\
\sC&\mapsto \sC^\Delta:=\Funct(\Delta,\sC).
\end{align*}

In particular, if $v:\sD\leftrightarrows \sC:u$ is an adjunction, then $v^\Delta:\sD^\Delta\leftrightarrows \sC^\Delta:u^\Delta$ is also an adjunction, because the adjunction identities for $(v,u)$ yield adjunction identities for $(v^\Delta,u^\Delta)$.

This also applies to pro-adjoints as follows: let $u:\sC\to \sD$ have a pro-left adjoint $v:\sD\to \pro{}\sC$. Equivalently \cite[Prop. A.2.1]{motmod}, $\pro{}u$ has a left adjoint $\tilde v$. Hence $(\pro{}u)^\Delta:(\pro{}\sC)^\Delta\to (\pro{}\sD)^\Delta$ has the left adjoint $\tilde v^\Delta$, which by composition yields a functor
\begin{equation}\label{eq1.0}
v^\Delta:\sD^\Delta\to (\pro{}\sC)^\Delta
\end{equation}
which verifies the same identity as in \cite[Prop. A.2.1 (ii)]{motmod}, and is actually the image of $v$ under the $2$-functor $\Delta_*$. In particular, $v^\Delta$ is computed by simply applying $v$ to the relevant diagrams.

\subsection{Diagrams of pro-objects} Keep the above notation. There is an obvious functor
\begin{equation}\label{eq1.1}
P:\pro{}(\sC^\Delta)\to(\pro{}\sC)^\Delta
\end{equation}
which sends  a pro-object $(\phi_i)_{i\in I}$ in $\sC^\Delta$ to the diagram $\delta\mapsto (\phi_i(\delta))_{i\in I}$.

\begin{prop}\label{p1.1} a) If $Ob(\Delta)$ is finite, \eqref{eq1.1} is faithful.\\
b) If $\Delta$ is finite, \eqref{eq1.1} is full.\\
c) If moreover $\Delta$ has no loops, \eqref{eq1.1} is an equivalence of categories.
\end{prop}

\begin{proof} Let $\phi=(\phi_i)_{i\in I}, \psi=(\psi_j)_{j\in J}$ be two objects of $\pro{}(\sC^\Delta)$. A morphism $\theta:\phi\to \psi$ is represented by a collection $(\theta_{j,i(j)})_{j\in J}$ where, for each $j$, $\theta_{j,i(j)}\in \sC^\Delta(\phi_{i(j)},\psi_j)$ for a suitable $i(j)$. Then $P(\theta)$ is the collection of diagrams of morphisms corresponding tautologically to these morphisms of diagrams. Consider another $\theta':\phi\to \psi$ and take a corresponding collection $(\theta'_{j,i'(j)})_{j\in J}$. Suppose that $P(\theta)=P(\theta')$. Then, for all $j\in J$ and all $\delta\in Ob(\Delta)$, there exists $i(j,\delta)\ge i(j), i'(j)$ such that $\theta_{j,i(j,\delta)}(\delta)=\theta'_{j,i(j,\delta)}(\delta)$, where $\theta_{j,i(j,\delta)}$ is the composition
\[\phi_{i(j,\delta)}\to \phi_{i(j)}\by{\theta_{j,i(j)}}\psi_j\]
and similarly for $\theta'_{j,i(j,\delta)}$. If $Ob(\Delta)$ is finite, we may choose $i(j,\delta)$ independent of $\delta$. This shows that $\theta=\theta'$, hence a).

Let $\rho:P(\phi)\to P(\psi)$ be a morphism. By definition, $\rho$ is a morphism between the functors $P(\phi), P(\psi):\Delta\to \pro{}\sC$, given by its components $\rho(\delta)$ subject to commutation with the morphisms of $\Delta$. Let $j\in J$; for each $\delta$, we have a morphism $\rho_{i(j,\delta),j}:\phi_{i(j,\delta)}(\delta)\to \psi_j(\delta)$ for a suitable $i(j,\delta)\in I$; as above, if $Ob(\Delta)$ is finite we may choose $i(j,\delta)$ independent of $\delta$, i.e. $i(j,\delta)=:i(j)$. If $\lambda:\delta_1\to \delta_2$ is a morphism, the naturality of $\rho$ with respect to $\lambda$ may be expressed in terms of the $\rho_{i(j),j}$ at the cost of replacing $i(j)$ by a larger, suitable $i'(j,\lambda)$; if $\Delta$ is finite, we may choose $i'(j,\lambda)$ independent of $\lambda$. Then $\rho$ is of the form $P(\tilde \rho)$, hence b).

Finally, c) (essential surjectivity) follows from \cite[Appendix, Prop. 3.3]{am}.
\end{proof}

\subsection{A pro-adjoint with parameters} We continue to suppose that $\Delta$ is finite and without loops. By Proposition \ref{p1.1}, the functor of \eqref{eq1.0} refines to a functor:
\begin{equation}\label{eq1.2}
w:\sD^\Delta\to \pro{}(\sC^\Delta).
\end{equation}

\begin{lemma}\label{l1} The functor \eqref{eq1.2} is pro-adjoint to $u^\Delta$.\qed
\end{lemma}

We shall now give an explicit and direct construction of this pro-adjoint, making Proposition  \ref{p1.1} possibly unnecessary. We give ourselves a system of subcategories $(I(d)\subset d\downarrow u)_{d\in \sD}$  representing $v$, as after Lemma \ref{l0}. Let $\ul{d}\in \sD^\Delta$.  %functor $i:\Delta\to Id_\sD\downarrow u$ such that, for each $\delta\in Ob(\Delta)$, $i(\delta)\in I'(c_\delta)$. A morphism in $I(\ul{c})$ is a morphism of functors. If $i\in I(\ul{c})$, we define $\ul{d}(i)=$

\begin{lemma}\label{l2} Define a subcategory $I(\ul{d})$ of $\ul{d}\downarrow u^\Delta$ as follows: an object $X$ (resp. morphism $f$) of $\ul{d}\downarrow u^\Delta$ is in $I(\ul{d})$ if and only if $X(\delta)$ (resp. $f(\delta)$) is in $I(d(\delta))$ for all $\delta\in \Delta$. Then:\\
a) The category $I(\ul{d})$ is ordered and cofiltering for all $\ul{d}\in \sD^\Delta$.\\ 
b) If $I(d)$ is full in $d\downarrow u$ for all $d\in \sD$, then $I(\ul{d})$ is full in $\ul{d}\downarrow u^\Delta$ for all $\ul{d}\in \sD^\Delta$.\\
c) The family $(I(\ul{d}))_{\ul{d}\in \sD^\Delta}$ corepresents  \eqref{eq1.2}.
\end{lemma}

\begin{proof} a) Ordered is obvious. For cofiltering, induction on $\# Ob(\Delta)$. We may assume $\Delta$ nonempty. The finiteness and ``no loop'' hypotheses imply that $\Delta$ has an object $\delta_0$ such that no arrow leads to $\delta_0$: we call such on object \emph{minimal}. Let $\Delta'$ be the subcategory of $\Delta$ obtained by removing $\delta_0$ and all the arrows leaving from $\delta_0$.  Let $X_1:\ul{d}\to u^\Delta(\ul{c_1})$, $X_2:\ul{d}\to u^\Delta(\ul{c_2})$ be two objects of $I(\ul{d})$. By induction, we may find $Y_3:\ul{d}\mid \Delta'\to u^{\Delta'}(\ul{c_3}')\in I(\ul{d}\mid \Delta')$ sitting above $X_1\mid \Delta'$ and $X_2\mid \Delta'$.  Let $f:\delta_0\to \delta$ be an arrow, with $\delta\in \Delta'$: by the functoriality of $v$, there exists a commutative diagram in $\sD$
\[\begin{CD}
d(\delta_0)@>\phi(f)>> u(c(f))\\
@V{d(f)}VV @Vu(\psi(f))VV\\
d(\delta)@>Y_3(\delta)>> u(c'_3(\delta))
\end{CD}\]
with $\phi(f)\in I(d(\delta_0))$. Since $I(d(\delta))$ is cofiltering, we may find an object $d(\delta_0)\by{g}u(c)\in I(d(\delta))$ sitting above all $\phi(f)$'s as well as  $X_1(\delta_0)$ and $X_2(\delta_0)$. Then, together with $Y_3$,  $X_3(\delta_0)=:g$ completes the construction of $X_3$ dominating $X_1$ and $X_2$.

b) is obvious. c) Let $\psi:\ul{d}\to u^\Delta(\ul{c})$ be a morphism in $\sD^\Delta$, with $\ul{c}\in \sC^\Delta$. To construct a morphism $\phi:(I(\ul{d}),pr_2)\to \ul{c}$ in $\pro{}(\sC^\Delta)$, we proceed as in a). Suppose $\phi\mid \Delta'$ has already been constructed. By definition, it means that an object $\ul{d}\mid \Delta'\to u^{\Delta'}(\ul{c_1})$ in $I(\ul{d}\mid \Delta')$ and a compatible morphism $\ul{c_1}\to \ul{c}\mid \Delta'$ have been given. Exactly as in a), we can complete this to a compatible pair $(\ul{d}\to u^{\Delta}(\ul{c_2}),\ul{c_2}\to \ul{c})$ yielding $\phi$. 
\end{proof}

\begin{rk}\label{r1}  Suppose that $\sC$ is essentially small and has finite limits. Then a functor $u:\sC\to \sD$ has a pro-adjoint if and only if it commutes with finite limits \cite[App., Cor. 2.6]{am}. If this is the case, then $\sC^\Delta$ and $u^\Delta$ verify the same hypotheses as soon as $\Delta$ is essentially small, and the existence of the pro-adjoint \eqref{eq1.2} is automatic. However,  in the case of Proposition \ref{p1} below, $\sC=\MCor$ does not have finite limits, for example no kernels (equalisers of two arrows) in general. Indeed, if it had kernels, so would $\Cor$, since $\omega:\MCor\to \Cor$ has a pro-adjoint. But it is easy to give examples of nonrepresentable kernels in $\Cor$, for example the equaliser of the two morphisms $f,g:\A^2\to \A^1$ given by $f(x,y)=y^2$ and $g(x,y)=x^3$.  So it seems that the passage through Proposition \ref{p1.1} (and in particular the condition on $\Delta$) is necessary in this case. The hypothesis on finite limits in $\sC$ is dropped from \cite[I.8.11.4]{SGA4}.
\end{rk}

\subsection{Cofinality}

Keep the situation of Lemma \ref{l2} a), and let $\Delta_1$ be a full subcategory of $\Delta$. We have an obvious functor
\[\phi:I(\ul{d})\to I(\ul{d}\mid \Delta_1).\]

The following lemma gives an abstract version of \cite[Lemma 4.3.1]{motmod}:

\begin{lemma}\label{l3} Suppose that $\Delta(\delta,\delta_1)=\emptyset$ for all $(\delta,\delta_1)$ such that $\delta\in \Delta-\Delta_1$ and $\delta_1\in \Delta_1$ (in the terminology of \cite[Def. 2.3.2]{quillen-dedekind}, the inclusion $\Delta'\subseteq \Delta$ is cellular). Then $\phi$ is cofinal.
%\\ More precisely, let $X\in I(\ul{d})$ and let $Y\in I(\ul{d}\mid \Delta_1)$, mapping to $\phi(X)$. Then there exists $X'\in I(\ul{d})$ mapping to $X$ and such that $\phi(X')=Y$.
\end{lemma}

\begin{proof} Consider $\Delta-\Delta_1$ as a full subcategory of $\Delta$. Let $\delta_0\in \Delta-\Delta_1$ be a minimal object (with respect to $\Delta-\Delta_1$) in the sense of the proof of Lemma \ref{l2} a): the hypothesis on $\Delta_1$ shows that  $\delta_0$ is also minimal with respect to $\Delta$. Setting $\Delta'=\Delta-\{\delta_0\}$ as in this proof, we are reduced by induction to the case where $\Delta_1=\Delta'$. We conclude with the same reasoning as in the proof of Lemma \ref{l2} a).
\end{proof}


\begin{thebibliography}{EGA4-IV}
\bibitem{am} M. Artin, B. Mazur \'Etale homotopy, Lect. Notes in Math. {\bf 100}, Springer, 1969.
\bibitem{Krishna-Park} A. Krishna, J. Park {\it Moving lemma for additive higher Chow groups}, Alg. Number Theory, {\bf 6} (2012), 293--326.
\bibitem{quillen-dedekind} B. Kahn {\it Around Quillen's theorem A}, \url{https://arxiv.org/abs/1108.2441}.
\bibitem{2017.1} B. Kahn, R. Sujatha {\it Birational motives, II: triangulated birational
motives}, IMRN {\bf 2017} (22), 6778--6831.
\bibitem{motmod} B. Kahn, S. Saito, T. Yamazaki {\it Motives with modulus}, preprint, 2017, rev. 2018, \url{https://arxiv.org/abs/1511.07124}.
\bibitem{mcl} S. Mac Lane Categories for the working mathematician, Grad. Texts in Math. {\bf
5}, Springer (2nd ed.), 1998.
\bibitem{neronmod} S. Bosch, W. L\"utkebohmert, M. Raynaud {\it N\'eron models}, Erg. Math. u. Grenzgebiete (3)  {\bf 21}, Springer, 1990.
\bibitem{mvw} C. Mazza, V. Voevodsky, C. Weibel {\it Lecture notes on motivic cohomology}, Clay Math. Monographs {\bf 2}. AMS -- Clay Math. Inst., 2006.
\bibitem{cubeinv} H. Miyazaki {\it Cube invariance of higher Chow groups with modulus}, \url{https://arxiv.org/abs/1604.06155}, to appear in J. Alg. Geometry.
\bibitem{RG} L. Gruson, M Raynaud {\it Crit\`eres de platitude et de projectivit\'e. Techniques de `platification' d'un module,} Invent. math.  {\bf 13} (1971), 1--89.
\bibitem{lazard} D. Lazard {\it Autour de la platitude}, Bull. SMF {\bf 97} (1969), 81--128.
\bibitem{oda} S. Oda {\it On finitely generated birational flat extensions of integral domains}, Ann. Math. Blaise Pascal {\bf 11} (2004), 35--40.
\bibitem{voetri} V. Voevodsky {\it Triangulated categories of motives over
a field}, {\it in} Cycles, transfers and motivic cohomology theories,
Annals of  Math. Studies {\bf 143}, Princeton University Press, 2000, 188--238.
\begin{center}
\sc Acronyms
\end{center}
\bibitem[EGA1]{EGA1} A. Grothendieck, J.A. Dieudonn\'e \'El\'ements de g\'eom\'etrie alg\'ebrique, I, Grundl. math. Wiss. {\bf 166}, Springer, 1971. 
\bibitem[EGA2]{EGA2} A. Grothendieck {\it \'El\'ements de g\'eom\'etrie alg\'ebrique : II. \'Etude globale \'el\'ementaire de quelques classes de morphismes} (EGA2), Publ. math. I.H.\'E.S. {\bf 8} (1961), 5--222.
\bibitem[EGA4-IV]{EGA4} A. Grothendieck {\it \'El\'ements de g\'eom\'etrie alg\'ebrique : IV. \'Etude locale des sch\'emas et des morphismes de sch\'emas} (EGA4), quatri\`eme partie, Publ. math. I.H.\'E.S. {\bf 32} (1967), 5--361.
\bibitem[SGA4-I]{SGA4} E. Artin, A. Grothendieck, J.-L. Verdier Th\'eorie
des topos et cohomologie \'etale des sch\'emas (SGA4), Vol. 1, Lect. Notes
in Math. {\bf 269}, Springer, 1972.
\end{thebibliography}
\end{document}